 \documentclass[11pt]{article}
 
\usepackage{amssymb,amsmath,amsfonts}
\usepackage{graphicx,color,enumitem}
\usepackage{mathrsfs}
\usepackage{amsthm} % For proof environment
\usepackage[dvipsnames]{xcolor}
\usepackage{bm}
\usepackage[round]{natbib}
\usepackage[]{appendix}

\usepackage[a4paper, total={6in, 9in}]{geometry}

\definecolor{mycolor}{rgb}
{0.796, 0.255, 0.329}

\RequirePackage[colorlinks,citecolor=blue,urlcolor=blue, linkcolor=red]{hyperref}

\usepackage{tikz}
\tikzstyle{vertex}=[circle, draw, inner sep=2pt, fill=white]%, minimum size=6pt] 

\newcommand{\s}{{\mathcal S}}
\newcommand{\e}{{\varepsilon}}

\newcommand{\E}{{\mathbb E}}

\renewcommand{\P}{{\mathbb P}}

\newcommand{\R}{{\mathbb R}}

\newcommand{\N}{{\mathbb N}}

\newcommand{\Acal}{{\mathcal A}}

\newcommand{\Dcal}{{\mathcal D}}
\newcommand{\Ecal}{{\mathcal E}}
\newcommand{\Fcal}{{\mathcal F}}
\newcommand{\Gcal}{{\mathcal G}}
\newcommand{\Hcal}{{\mathcal H}}
\newcommand{\Ical}{{\mathcal I}}

\newcommand{\Lcal}{{\mathcal L}}
\newcommand{\Ucal}{{\mathcal U}}
\newcommand{\Mcal}{{\mathcal M}}

\newcommand{\Rcal}{{\mathcal R}}
\newcommand{\Scal}{{\mathcal S}}

\newcommand{\Nu}{{N}}

\newcommand{\Id}{{\textup{Id}}}

\newcommand{\B}{{B}}
\newcommand{\aaa}{{a}}
\newcommand{\A}{{a}}

\newcommand{\M}{{M}}
\newcommand{\y}{{y}}
\newcommand{\z}{{z}}
\newcommand{\cc}{{\overline c}}

\newcommand{\Span}{\textup{span}}
\newcommand{\ootimes}{\underline\otimes}

\newcommand{\fdot}{{\,\cdot\,}}

\newtheorem{theorem}{Theorem}

\newtheorem{corollary}[theorem]{Corollary}

\newtheorem{definition}[theorem]{Definition}
\newtheorem{remark}[theorem]{Remark}
\newtheorem{example}[theorem]{Example}

\newtheorem{lemma}[theorem]{Lemma}

\newtheorem{proposition}[theorem]{Proposition}

\numberwithin{equation}{section}
\numberwithin{theorem}{section}

\begin{document}
\date{}
\title{Infinite dimensional polynomial processes}
\author{Christa Cuchiero\thanks{Vienna University of Economics and Business, Institute of Statistics and Mathematics, Welthandelsplatz 1, A-1020 Wien, Austria, christa.cuchiero@wu.ac.at}
\and Sara Svaluto-Ferro\thanks{Faculty of Mathematics, Vienna University, Oskar-Morgenstern-Platz 1, A-1090 Wien, Austria, sara.svaluto-ferro@univie.ac.at.\newline
The authors gratefully acknowledge financial support by the Vienna Science and Technology Fund (WWTF) under grant MA16-021.
}}

\maketitle

\begin{abstract}
We introduce polynomial processes taking values in an arbitrary Banach space $\B$ via their infinitesimal generator $L$ and the associated martingale problem. 
We obtain two representations of the (conditional) moments 
in terms of solutions of a system of ODEs 
on the truncated tensor algebra of dual respectively bidual spaces. 
We illustrate how the well-known moment formulas for finite dimensional or probability-measure valued polynomial processes can be deduced in this general framework. 
As an application we consider polynomial forward variance curve models which appear in particular as Markovian lifts of (rough) Bergomi-type volatility models. Moreover, we show that the signature process of a $d$-dimensional Brownian motion is polynomial and derive its  expected value via the polynomial approach. 

\end{abstract}


%\tableofcontents

\section{Introduction}

Polynomial processes in finite dimensions, introduced in \cite{CKT:12} (see also
\cite{FL:16}), constitute a class of time-homogeneous Markovian It\^o semimartingales which are inherently tractable: conditional moments
can be expressed through  a \emph{deterministic dual process} which is the solution of a linear ODE. This is the so-called \emph{moment formula}. They form a rich class that includes Wright-Fisher diffusions (\cite{K:64}) from population genetics, Wishart correlation matrices (\cite{AA:13}), and affine processes (\cite{DFS:03}), just to name a few. The computational advantages due to the moment formula are obvious and  have led to a wide range of applications,  in particular in mathematical finance and population genetics. 
In mathematical finance, this concerns especially interest rate theory, stochastic volatility models, life insurance liability modeling, variance swaps, and stochastic portfolio theory
(see, e.g., \cite{AFP:16}, \cite{BZ:16}, \cite{FGM:16}, \cite{C:19}). 
In population genetics, dual processes associated to moments 
and simple procedures to compute them play an equally important role: the Wright-Fisher diffusion with seed-bank component (see, e.g., \cite{BCKW:16} and the references therein) is for instance an important example of a recently investigated two-dimensional polynomial process in this field.

The main goal of this paper is to introduce the concept of  \emph{polynomial processes on a general Banach space $B$} and to provide a corresponding \emph{moment formula}. This leads in particular 
to a comprehensive theory covering practically \emph{all} finite and infinite dimensional state spaces, hence far beyond the specific cases considered so far. The second main goal is to illustrate the applicability of the moment formula and the powerful and easy-to-use results  which arise therefrom. The potential of this formula can be appreciated by means of several practically relevant examples. 

Let us here first focus on examples from mathematical finance. The most recent
appearance of infinite dimensional polynomial processes is certainly in the field of \emph{rough volatility}  (see, e.g., \cite{ALV:07, GJR:18, BLP:16}).
This roughness can be seen as a generic non-Markovianity  of the corresponding volatility processes. By lifting these processes to \emph{infinite dimensions,} it is however possible to recover the Markov property.
 Intriguingly such infinite dimensional models often stem from the class of polynomial processes: Hawkes processes, rough Heston (\cite{ER:19}), and rough Wishart processes can be viewed as infinite dimensional affine and thus polynomial processes as shown in \cite{AE:19, CT:18, CT:19}. In the current article we show that the rough Bergomi model (\cite{BFG:16}) also pertains to the class of infinite dimensional polynomial processes. In other areas, like in stochastic portfolio theory the most flexible and tractable models appear again to be (measure-valued) polynomial (\cite{CLS:18}) and also the Zakai equation from filtering theory belongs to this class. Let us here also mention that the (sub)class of affine processes taking values in Hilbert spaces has recently been studied in \cite{STY:19}, in particular from an existence and pathwise uniqueness point of view.

In population genetics, infinite dimensional models appear in form of the well-known measure-valued diffusions such as the Fleming--Viot process, the Super-Brownian motion, and the Dawson--Watanabe superprocess (see, e.g., \cite{E:00} and the references therein). All these examples are polynomial processes.

The deeper reason behind this predominance can be explained by a universal approximation
property of  polynomial dynamics in the space of all stochastic dynamics driven by say, Brownian motion (or many other continuous processes). This is based on the properties of the
\emph{signature process,}  which plays a prominent role in rough path theory introduced by \cite{L:98} and which serves as a regression basis for solutions of  general stochastic differential equation (see, e.g., \cite{LLN:13}).
 As the signature of many processes, in particular of $d$-dimensional Brownian motion,  also turns out to be an infinite dimensional polynomial process, this suggests an inherent universality of the polynomial class (see Section \ref{sec:sig} for more details).
We refer also to~\cite{CGGOT:19} where a randomized polynomial signature process is used as regression basis.

Passing to infinite dimensional polynomial models is also supported 
from a purely computational point of view.  
Indeed, the synergy of increasing computer power with machine learning techniques has enabled to treat high dimensional linear PDEs (or infinite dimensional linear ODEs), exactly the kind of equation that arises in our context,  very efficiently, for instance via neural network approaches (see e.g.~\cite{BBGJJ:18}).
 Beside these techniques, the nice symmetries that typically characterize polynomial processes allow to employ highly efficient algorithms, as those proposed by \cite{HPR:18}.

Let us now outline our approach:  we fix a state space $\Scal\subseteq B$ and define \emph{polynomial processes} as $\mathcal{S}$-valued  solutions of martingale problems for certain linear operators $L$, which we coin \emph{polynomial} as well. For this class of processes we derive representations of the moments in terms of two moment formulas.

More precisely, we first introduce polynomials on $B$ following the definition of polynomials in finite dimensions. In this spirit, we first define homogeneous polynomials of degree $k$ to be bounded linear maps of $\y^{\otimes k}$, where $\y^{\otimes k}$ denotes the $k$-fold    tensor product of $\y\in \B$ with itself. Then, we define polynomials on $B$ of degree $k$ as linear combinations of homogeneous polynomials on $B$ of degree less or equal to $k$. That is, a \emph{polynomial on $\B $ with coefficients $\A_0\in (\B^{\otimes 0})^*,\ldots, \A_k\in (\B^{\otimes k})^*$} is defined as
\begin{equation*}
p(\y)=\A_0+\langle \A_1,\y\rangle+\ldots+\langle \A_k,\y^{\otimes k}\rangle,
\end{equation*}
 where $\B^{\otimes j}$ denotes the $j$-fold symmetric
  algebraic tensor product of $\B$ equipped with some crossnorm, $(B^{\otimes j})^*$ denotes the dual space of $B^{\otimes j}$, and $\langle \fdot, \fdot \rangle$ denotes the pairing between $B^{\otimes j}$ and $(B^{\otimes j})^*$.

We then define polynomial operators as linear operators acting on classes of cylindrical polynomials, i.e.~functions $p$ of the form
\[
 p(y)=\phi\left(\langle a_1, \y\rangle, \ldots, \langle a_d, \y\rangle \right),
\]
where $\phi$ is a polynomial on $\mathbb{R}^d$ and $a_1,\ldots,a_d \in D$ $\subseteq B^*$. Their defining property consists in mapping each polynomial $p$ to a (not necessarily cylindrical) polynomial $Lp$ without increasing the degree. 
Polynomial processes on $\mathcal{S} \subseteq B$ are then defined to be $\Scal$-valued solutions of martingale problems for polynomial operators.

This allows to associate to $L$  two families of linear operators,  $(L_k)_{k \in \mathbb{N}}$ and $(M_k)_{k \in \mathbb{N}}$, that we call \emph{dual} and \emph{bidual}, respectively. Those names are mnemonic for the spaces where those linear operators act. Indeed, $L_k$  maps  the coefficients vector of a cylindrical polynomial $p$ to the coefficients vector of the polynomial $Lp$, and is thus defined on  the truncated graded algebra $\bigoplus_{j=0}^k (\B^{\otimes j})^*$. 
The bidual operator $M_k$  is the adjoint operator of $L_k$ when pairing $\bigoplus_{j=0}^k (\B^{\otimes j})^*$ with $\bigoplus_{j=0}^k (\B^{\otimes j})^{**}$ and is thus defined on the truncated graded algebra $\bigoplus_{j=0}^k (\B^{\otimes j})^{**}$.
To both operators we can associate a system of $k+1$ dimensional ODEs on the respective truncated algebras. The corresponding solution will then  describe the evolution of the polynomial process' moments
(see Theorem~\ref{thm1} and Theorem~\ref{dual2}).

Note that our approach differs from the inspiring paper by~\cite{BDK:18} on the related notion of \emph{multilinear processes} in Banach spaces in the following sense. While in~\cite{BDK:18} the moment formula itself is the  defining property and examples which satisfy it, e.g.~independent increment processes, are studied, we tackle the problem from a different angle. Indeed, we show systematically that for Banach space valued processes defined via polynomial operators the moment formula holds true.

 We then apply this abstract theory to specific Banach spaces.
For finite dimensions we show that the operator $L_k$ corresponds to the matrix representation (with respect to a certain basis) of the infinitesimal generator $L$ restricted to the set of all polynomials up to degree $k$. The operator  $M_k$ is then simply the transpose of this matrix and the moment formulas boil down to the well known matrix exponential representation of the conditional moments of finite dimensional polynomial processes  (see \cite{CKT:12} and \cite{FL:17} for more details). Turning then to an infinite dimensional setting, we show that for the state space  of probability measures on $\mathbb{R}$, the linear ODEs corresponding to $L_k$ and $M_k$ can be identified with the Kolmogorov  backward equation and the Kolmogorov forward equation, respectively, of an associated $\mathbb{R}^k$-valued Markov process.
 As our main application, we introduce a class of 
 polynomial forward variance curve models which take values in a Hilbert space as in \cite{F:01}. These are SPDEs with polynomial characteristics describing the evolution of the forward variance curve, that is $x \mapsto \mathbb{E}[V_{t+x}|\mathcal{F}_t]$ where $V$ denotes the spot variance of some asset. This setup includes
for instance the (rough) Bergomi model as considered in  \cite{B:05, BFG:16, JMM:18}.
  We show in particular how to exploit the moment formula to price options on VIX in such models.  
 As last example  we illustrate that the signature process of a $d$-dimensional Brownian motion is polynomial and derive its  expected value (see, e.g., \cite{FH:14}) via the polynomial approach. 

The remainder of the article is organized as follows: In Section \ref{IIs1} we introduce some basic notation and definitions. In Section \ref{IIsec:2} we define  polynomials on $B$ and polynomial operators. Section \ref{sec_ex_un} is devoted to the introduction of polynomial processes as well as the formulation and proofs of the moment formulas.
We then conclude with Section \ref{sec:Examples}, which covers several examples and applications, including generic polynomial operators of L\'evy  type, probability measure valued polynomial processes, forward variance models, and the signature process of a $d$-dimensional Brownian motion.

\subsection{Notation and basic definitions}\label{IIs1}

Throughout this paper let $\B$ be a  Banach space. We denote by $\B^*$ its dual space, i.e.~the space of linear continuous
functionals with the strong dual norm
\[ 
{\|\A\|}_{*} = \sup_{\|\y\| \leq 1} |\A(\y)|.
\]
We introduce now some basic notions in the context of tensor products. For more details we refer to \cite{R:13}.
For two elements $\y_1,\y_2\in \B$ we denote by $\y_1\otimes \y_2$ the corresponding algebraic symmetric tensor product  and by
$$\B\otimes \B:=\Big\{\sum_{i=1}^d\alpha_i \y_i\otimes \y_i\ : \ \alpha_i\in \R,\ \y_i\in \B,\ d\in \N\Big\}$$
the algebraic tensor product of $\B$ with itself. For an element  $\A\in \B^*$ the linear map $\A\otimes \A:\B\otimes \B\to\R$ is defined by
$$\A\otimes \A(\y_1\otimes \y_2)=\A(\y_1)\A(\y_2).$$
For $\A_1, \A_2\in \B^*$ we then get  $\A_1\otimes \A_2$ via polarization. 
With $\|\fdot\|_\times$ we denote a crossnorm on $\B\otimes \B$, i.e.~a norm that satisfies
\begin{enumerate}
\item $\|\y_1\otimes \y_2\|_\times=\|\y_1\|\|\y_2\|$ for each $\y_1,\y_2\in \B$, and
\item\label{ii} $\sup_{\y\in \B\otimes \B, \|\y\|_\times\leq1}|\A_1\otimes \A_2(\y)|= \|\A_1\|_{*}\|\A_2\|_{*}$ for each $\A_1,\A_2\in \B^*$.
\end{enumerate}
Throughout we shall fix some crossnorm with respect to which we define the corresponding dual spaces.
We then denote by $(\B\otimes \B)^*=(B^{\otimes 2})^*$ the  dual of space of $(\B\otimes \B,\|\fdot\|_\times)$.
Observe that $\A_1\otimes \A_2\in(\B\otimes \B)^*$ for each $\A_1,\A_2\in \B^*$.
Furthermore, we denote by $\|\fdot\|_{*2}$ the strong dual norm on $(\B\otimes \B)^*$, i.e.
\begin{align}\label{eqn4}
\|\A\|_{*2}:=\sup\{|\A(\y) |:\ \y\in \B\otimes \B, \|\y\|_\times\leq1\}.
\end{align}
The $k$-fold tensor products $\y^{\otimes k}, \B^{\otimes k}, \A^{\otimes k}, (\B^{\otimes k})^*$ and the norm $\|\fdot\|_{*k}$ for $k\in\N$ are defined analogously. For $k=0$, we identify $(\B^{\otimes 0})^*$ with $\R$, such that $\A(\y^{\otimes 0}):=\A\in \R$ for all $\y\in \B$. We also set $y^{\otimes 0}:=1$.
Generally, for subsets in $D_j \subseteq (\B^{\otimes j})^*$ or $\Scal_j\subseteq(\B^{\otimes j})$  we write $\vec\A\in \bigoplus_{j=0}^k D_j$ for $\vec\A=(\A_0,\ldots,\A_k)$ with  $\A_j\in D_j$ and  $\vec \y\in \bigoplus_{j=0}^k \Scal_j$ for $\vec \y=(\y_0,\ldots,\y_k)$ with  $\y_j\in \Scal_j$. Moreover, for $\y\in\B$ we write $\overline \y$ for the vector $(1,\y,\ldots,\y^{\otimes k})$.

\section{Polynomials on $\B$ and polynomial operators} \label{IIsec:2}

The goal of this section is to introduce polynomials on $B$ and to define all sorts of operators that we shall need when dealing with polynomial processes.

For $\A\in (\B^{\otimes k})^*$, $\y\in \B$, and $k\in \N_0$ we use the notation
$\langle \A,\y^{\otimes k}\rangle:=\A(\y^{\otimes k}).$
A  \emph{polynomial on $\B $ with coefficients $\A_0\in (\B^{\otimes 0})^*,\ldots, \A_k\in (\B^{\otimes k})^*$} is then defined as
\begin{equation}\label{eqn2}
p(\y)=\A_0+\langle \A_1,\y\rangle+\ldots+\langle \A_k,\y^{\otimes k}\rangle.
\end{equation}
Setting $\vec\A:=(\A_0,\ldots,\A_k)\in\bigoplus_{j=0}^k  (\B^{\otimes j})^*$ we say that \emph{$\vec\A$ is the coefficients vector corresponding to $p$}. The degree of a polynomial $p(\y)$, denoted by $\deg(p)$, is the largest $j$ such that $\A_j$ is not the zero function, and $-\infty$ if $p$ is the zero polynomial.
As short hand notation, we shall often denote $p(\y)$ via
\begin{equation}\label{eqn501}
p(\y)= (\vec a \cdot \overline{\y}).
\end{equation}
Recall here that $\overline \y$ stands for the vector $(1,\y,\ldots,\y^{\otimes k})$.
More generally, we shall write 
\begin{align} \label{eq:dotrel}
(\vec\A\cdot\vec {\y}):=\A_0+\langle \A_1,\y_1\rangle+\ldots+\langle \A_k,\y_k\rangle
\end{align} 
for $\vec \A \in \bigoplus_{j=0}^k (B^{\otimes j})^* $  and $\vec \y \in \bigoplus_{j=0}^k (B^{\otimes j})^{**}$.
Next, we denote by
$$P:=\{\y\mapsto p(\y)\ :\ p\text{ is a polynomial on }\B \}$$
 the algebra of all polynomials on $\B $ regarded as real-valued maps, equipped with the pointwise addition and multiplication.
In practice, it is often convenient to consider a subspace of polynomials with more regular coefficients. Fix $D\subseteq \B^*$.
In our context
a \emph{cylindrical polynomial with coefficients in $D$} is a function $p: \B  \to \mathbb{R}$ of the form 
$
p(\y):= \phi(\langle \A_1, \y \rangle,\ldots,\langle \A_d, \y \rangle), 
$
where $d\in\N$, $\phi:\R^d\to\R$ is a polynomial, and $\A_i\in D$ for each $i\in\{1,\ldots,d\}$. 
The space of cylindrical polynomials with coefficients in $D$ is defined by
\[
P^D=\Span \{\y \mapsto \langle  \A, \y \rangle^k\, : k \in \mathbb{N}_0, \A \in D\}.
\]
Since for  $\A_1\in (\B^{\otimes {k_1}})^*$ and $\A_2\in (\B^{\otimes k_2})^*$ it holds 
$
\langle \A_1,\y^{\otimes k_1}\rangle\langle \A_2,\y^{\otimes k_2}\rangle=\langle \A_1\otimes \A_2,\y^{\otimes (k_1+k_2)}\rangle,
$
 we can equivalently write
$P^D=\{\y \mapsto (\vec \A \cdot \overline{\y})\, : k \in \mathbb{N}_0, \A_j \in D^{\otimes j}\}$.

 We now define polynomial operators, which constitute a class of possibly unbounded linear operators acting on polynomials. They are not defined on all of $P$ in general, but only on the subspace $P^D$ for some   subspace $D \subseteq \B^*$. An analog of this notion has appeared previously in connection with finite-dimensional and measure-valued polynomial processes; see e.g.~\cite{CKT:12,FL:16, CLS:18}.

\begin{definition}\label{IIdef1}
Fix $\s\subseteq \B$. A linear operator $L\colon P^D\to P$ is called $\Scal$-polynomial if for every $p\in P^D$ there is some $q\in P$ such that $q|_{\s}=Lp|_{\s}$ and
$$\deg(q) \le \deg(p).$$
\end{definition}

\begin{remark}
Recall from Section \ref{IIs1} that we chose  to work with the symmetric tensor product. The reason for this is that a non-symmetrized tensor product would lead to an unnecessary redundancy: since $\langle\A\otimes b,\y^{\otimes 2}\rangle=\langle \A,\y\rangle\langle b,\y\rangle=\langle b\otimes \A,\y^{\otimes 2}\rangle$ the polynomial with coefficient $\A\otimes b$ coincides with the one with coefficient $b\otimes \A$.
\end{remark}

\subsection{Dual operators}

 We would now like to associate to an $\Scal$-polynomial operator a family of so-called \emph{dual} operators  $(L_k)_{k \in \mathbb{N}}$ which are linear operators mapping the coefficients vector of $p$ to the coefficients vector of $Lp$.

 If $\Scal=\B$, the definition of those operators is easily achievable. Indeed, 
 in this case the representation provided in \eqref{eqn2} is unique (see Lemma~\ref{lem4} for more details) and we can identify each polynomial 
 with its coefficients vector.
  For a $\B $-polynomial operator $L$, this means that $L$  can be uniquely identified with 
 a family of operators $(L_k)_{k\in \N}$,  
 where $L_k: \bigoplus_{j=0}^k D^{\otimes j}\to\bigoplus_{j=0}^k (\B^{\otimes j})^*$ maps the coefficients vector of $p$ to the coefficients vector of $Lp$, 
 for each $p\in P^D$ with $\deg(p)\leq k$. By letting $L_k^j:  \bigoplus_{j=0}^k D^{\otimes j}\to (\B^{\otimes j})^*$ be the operator mapping the coefficients vector of $p$ to the $j$-th coefficient of $Lp$, we can write $L_k\vec\A=(L_k^0\vec\A,\ldots,L_k^k\vec\A)$ for all $\vec\A\in \bigoplus_{j=0}^k D^{\otimes j}$. It is important to note that the operators $L_k$ and $L_k^j$ inherit linearity from $L$.
 
 If $\Scal\subsetneq \B $, two different polynomials can coincide on $\Scal$. This is for instance the case if $\Scal=\{\y \in \B :\langle \A_1,\y\rangle=\langle \A_2,\y\rangle\}$ for some $\A_1 \neq \A_2 \in B^*$. In such a situation there is no one-to-one correspondence between the restriction of a polynomial to $\Scal$ and its coefficients vector. However, it is still possible to define a dual operator that maps the coefficients vector of $p$ to the coefficients vector of $q$, where $q$ is \emph{a}  polynomial such that $Lp|_\Scal=q|_\Scal$ and $\deg(q)\leq \deg(p)$. Since the choice of such a $q$ may not be unique, the linearity of such an operator is a priori not clear. We show in Lemma~\ref{lem12} that one can always choose $q$ in a such way that linearity  of the dual operator is satisfied. We can thus conclude that to each $\Scal$-polynomial operator $L$ we can associate a family of dual operators  $(L_k)_{k \in \mathbb{N}}$, rigorously introduced in the following definition. Recall that the paring $(\ \fdot\ )$ has been introduced in \eqref{eqn501}.

\begin{definition} \label{IIeq:L on Dk}
Let $L\colon P^D\to P$ be an $\Scal$-polynomial operator and fix $k\in \N$. 
A \emph{$k$-th dual operator} $L_k: \bigoplus_{j=0}^k D^{\otimes j}\to\bigoplus_{j=0}^k (\B^{\otimes j})^*$ is a linear operator such that $L_k\vec\A=:(L_k^0\vec\A,\ldots, L_k^k\vec\A)$ satisfies
$$
Lp(\y)=(L_k \vec \A  \cdot \overline {\y})\qquad \text{for all }\y\in \Scal,
$$
where $p(\y):=(\vec \A\cdot\overline \y)$. Whenever $L_k$ is a closable operator, we still denote its closure\footnote{We refer for instance to Chapter 1 in \cite{EK:05} for the precise definiton.} by $L_k: \Dcal(L_k)\to\bigoplus_{j=0}^k (\B^{\otimes j})^*$ and its domain by $\Dcal(L_k)\subseteq\bigoplus_{j=0}^k(\B^{\otimes j})^*$.
 \end{definition}

We illustrate this notion by means of the very well studied one-dimensional Jacobi diffusion.

\begin{example}\label{ex1}
Let $\B=\B^*=D=\R$, $\Scal:=[0,1]$, and recall that $\R\otimes \R=\R$. Let $P$ denote the space of all polynomials on $\R$ and let $L:P\to P$ be the polynomial operator given by
$$Lp(\y)=\y(1-\y)p''(\y).$$
For each $\vec\A:=(\A_0,\ldots,\A_k)\in \R^{k+1}$ fix then $p_{\vec\A}(\y):=\A_0+\A_1\y+\ldots+\A_k\y^k$ and compute
\begin{equation}\label{eqn7}
Lp_{\vec \A}(\y)=2\A_2\y+\ldots+(-j(j-1)\A_{j}+(j+1)j\A_{j+1})\y^{j}+\ldots+(-k(k-1)\A_k)\y^{k}.
\end{equation}
Observe that $Lp_{\vec \A}$ is again a polynomial of degree at most $k$, showing that $L$ is indeed $\R$-polynomial, and thus $[0,1]$-polynomial.
Moreover, from \eqref{eqn7} one can see that the $k$-th dual operator $L_k:\R^{k+1}\to\R^{k+1}$ corresponding to $L$ is given by
  $L_k^{j}\vec\A=-j(j-1)\A_{j}+(j+1)j\A_{j+1}$ for $j\in\{0,\ldots,k-1\}$ and $L_k^k\vec\A=-k(k-1)\A_k.$
$L_k$ can thus be identified with the unique matrix $G_k$ such that
$L_k\vec \A=G_k\vec\A$. 
\end{example}

\begin{remark}\label{rem3}
Observe that whenever a $k$-th dual operator $L_k$ satisfies $L_k(0,\ldots,D^{\otimes j},\ldots,0)\subseteq (0,\ldots,(\B^{\otimes j})^*,\ldots,0)$ for $j\in\{0, \ldots, k\}$, one can define auxiliary operators $\Lcal_j:D^{\otimes j}\to(\B^{\otimes j})^*$
 such that
$$L_k\vec \A=(\Lcal_0\A_0,\ldots, \Lcal_k\A_k)$$
for each $\vec \A\in  \bigoplus_{j=0}^k D^{\otimes j}$ and in turn also for each  $\vec \A\in (\Dcal(\Lcal_0),\ldots,\Dcal(\Lcal_k)):=  \Dcal(L_k)$.
\end{remark}

\subsection{Bidual operators}

Next we introduce the notion of a \emph{bidual} operator, which is a slightly more delicate.   Its domain of definition 
\begin{equation}\label{eqn3}
\Dcal_k:=\Span\bigg\{\overline \y\in \bigoplus_{j=0}^k\Scal^{\otimes j} \colon |Lp_{\A_i}(\y)|\leq C_\y \|\A_i\|_{*i}\text{ for all }\A_i\in D^{\otimes i}\text{ and } i\in\{0,\ldots,k\}\bigg\}
\end{equation}
where $p_{\A_i}(\y):=\langle \A_i,\y^{\otimes i}\rangle$, $C_\y\in \R$, and $\|\fdot\|_{*i}$ is given by \eqref{eqn4} (extended to higher order). In words, $\Dcal_k$ contains the linear combinations of all $\overline \y=(1,\y,\ldots,\y^{\otimes k})$ such that the linear operator from $(D^{\otimes i},\|\fdot\|_{*i})$ to $\mathbb{R}$ given by
$$\A_i\mapsto  
Lp_{\A_i}(\y)$$
is bounded (and thus continuous), for each $i\in\{0,\ldots,k\}$. Again, recall the pairing $(\ \fdot\ )$ defined in \eqref{eqn501} and $\overline \y:=(1,\y,\ldots,\y^{\otimes k})$.

\begin{definition}\label{def4}
Fix $k\in \N$ and let $\Dcal_k$ be the set defined in \eqref{eqn3}.
A \emph{$k$-th bidual operator}  $\M_k:\Dcal_k \to\bigoplus_{j=0}^k (\B^{\otimes j})^{**}$ is a linear operator 
satisfying 
$$
Lp( \y)= (\vec \A \cdot M_k \overline{y}),\qquad \text{for all }\vec \A\in \bigoplus_{j=0}^k D^{\otimes k}
$$
where $p(\y):=(\vec\A\cdot\overline y)$.  
\end{definition}

The motivation to work with the potentially large bidual spaces is given in Remark \ref{rem:motivation} below.
As one would expect, dual and bidual operators are strongly connected. We exploit this relation in the following lemma.
\begin{lemma}\label{lem1}
Fix a $k$-th dual operator $L_k$ and a $k$-th bidual operator $\M_k$. Then $L_k$ and $\M_k$ are adjoint with respect to the relation $(\ \fdot\ )$, defined in \eqref{eq:dotrel}, meaning that
\begin{align}\label{eqn1}
(L_k \vec \A \cdot \vec \y)=(\vec \A \cdot M_k \vec \y)
\end{align}
for each $\vec\A\in \bigoplus_{j=0}^k D^{\otimes k}$ and $\vec\y\in \Dcal_k$. Whenever $L_k$ is a closable operator, \eqref{eqn1} holds also for all $\vec \A\in \Dcal(L_k)$. 
\end{lemma}
\begin{proof}
The result follows by noting that $(L_k \vec \A \cdot \overline \y)=Lp_{\vec \A}( \y)=(\vec \A \cdot M_k \overline \y)$ for each $\vec\A\in\bigoplus_{j=0}^k D^{\otimes j}$ and $\y\in\Scal$ such that $\overline \y\in \Dcal_k$, where $p_{\vec\A}(\y):=(\vec\A\cdot\overline y)$.\end{proof}

To illustrate the notion of the bidual operator we consider again the one-dimensional Jacobi diffusion.

\begin{example}\label{ex2}
Consider again the setting of Example~\ref{ex1} and set again $\overline y:=(1,\y,\ldots,\y^k)$ for each $\y\in[0,1]$. An inspection of \eqref{eqn7} shows that
the $k$-th bidual operator $\M_k:\R^{k+1}\to\R^{k+1}$ corresponding to $L$ satisfies
$\M_k^i\overline \y=i(i-1)(\y^{i-1}-\y^i)$ for each $\y\in [0,1]$.
One thus gets  that
$\M_k^i\vec \y=i(i-1)(\y_{i-1}-\y_i)$ for $i\in\{0,\ldots,k\}$ and $\vec \y\in \R^{k+1}$.
As in the case of dual operators, the $k$-th bidual operator can be identified with the unique matrix $\widetilde G_k$ such that
$\M_k\vec\y=\widetilde G_k\vec y$ for all $\vec \y\in \R^{k+1}$. A direct computation shows then that $\widetilde G_k=G_k^\top$. This relation is nothing else than 
\eqref{eqn1}, which in this setting reads
$\vec\A^\top \widetilde G_k\vec \y= (\vec \A \cdot M_k \vec \y)= (L_k \vec \A \cdot \vec \y)
=\vec\A^\top G_k^\top\vec \y$
for all $\vec\A,\vec\y\in \R^{k+1}$.
\end{example}

\begin{remark}\label{rem33}
Suppose that the conditions of Remark~\ref{rem3} hold
for some $k$-th dual operator $L_k$, set $\Dcal_{k,j}:=\{\y_j\in \Scal^{\otimes j}\colon \vec \y\in \Dcal_k\}$, and consider the linear operator $\Mcal_j:\Dcal_{k,j}\to(\B^{\otimes j})^{**}$ uniquely defined by
$$\langle\A_j,\Mcal_j\y_j\rangle:=\langle \Lcal_j\A_j,\y_j\rangle,\qquad \A_j\in \Dcal(\Lcal_j),\ \y_j\in \Dcal_{k,j}.$$
By Lemma~\ref{lem1} we can then show that
$\M_k\vec \y=(\Mcal_0\y_0,\ldots, \Mcal_k\y_k)$
for each $\vec \y\in \Dcal_k$. 
Indeed, under the given conditions
$\langle \Lcal_j \A_j,\y^{\otimes j}\rangle=(L_j(\A_j\vec e_j)\cdot\overline y)=(\A_j \vec e_j \cdot M_j\overline y)$ for each $\y\in \B$.
By linearity we can conclude that the $j$-th component of  $M_j$ can depend on $\overline y$ just through $\y^{\otimes j}$, whence the above equality.

\end{remark}

\section{Polynomial processes on $\B $} \label{sec_ex_un}
In this section we  define  a $\B $-valued polynomial process, and derive two moment formulas.  We start by introducing a concept of measurability to which we implicitly always refer when speaking of $\B$-valued random variables and processes.

\begin{definition}\label{def3}Fix a filtered probability space $(\Omega, \Fcal,(\Fcal_t)_{t\geq0}, \P)$.
\begin{enumerate}
\item Fix $\Gcal\subseteq\Fcal$ and $k\in \N$. A map $\lambda:\Omega\to\B^{\otimes k}$ is $\Gcal$-weakly-measurable if $\langle \A,\lambda\rangle$ is $\Gcal$-measurable  for all $\A\in(\B^{\otimes k})^*$. For $\Gcal=\Fcal$ we  say that $\lambda$ is weakly-measurable.
\item
A $\B$-valued  adapted process $(\lambda_t)_{t\geq0}$ (or simply a $\B$-valued process) is a map defined on $\R_+\times \Omega$ with values in $\B$  such that $\lambda_t^{\otimes k}:\Omega\to\B^{\otimes k}$ is $\Fcal_t$-weakly-measurable for each $k\in \N$ and each $t\geq0$. 
\end{enumerate}
\end{definition}
In particular, for a $\B$-valued process  $(\lambda_t)_{t\geq0}$ one has that $(p(\lambda_t))_{t\geq0}$ is a real valued adapted process for all $p\in P$ and $t\geq0$.

 Let $\Scal\subseteq \B $, fix a linear subspace $D\subseteq \B^*$, and let $L: P^D \to P$ be a linear operator. 
An $\Scal$-valued process $(\lambda_t)_{t\geq0}$ defined on some filtered probability space $(\Omega, \Fcal,(\Fcal_t)_{t\geq0}, \P)$  is called a {\em solution to the martingale problem for $L$} with initial condition $\y_0 \in \Scal$ if 
\begin{enumerate}
\item $\lambda_0= \y_0$ $\P$-a.s., 
\item for every $p\in P^D$ there exists a  c\`adl\`ag version of $(p(\lambda_t))_{t \geq 0}$ and $(Lp(\lambda_t))_{t \geq 0}$  and
\item the process
\begin{equation}\label{IIeqnN}
N^p_t := p(\lambda_t) - p(\lambda_0) - \int_0^t Lp(\lambda_s) ds
\end{equation}
defines a local martingale for every $p\in P^D$. 
\end{enumerate}

Uniqueness of solutions to the martingale problem is always understood in the sense of law. The martingale problem for $L$ is {\em well--posed} if for every $\y\in \Scal$ there exists a unique $\Scal$-valued solution to the martingale problem for $L$ with initial condition~$\y_0$. 

\begin{definition}
Let $L$ be $\Scal$-polynomial. A solution to the martingale problem for $L$ is called  $\Scal$-valued  \emph{polynomial process}.
\end{definition}

\subsection{Dual moment formula }\label{sec_moments_uniqueness}
Our goal here is to  derive an analog of the moment formula in this general infinite dimensional setting.
To do this, it is crucial that the local martingales defined in \eqref{IIeqnN} 
are in fact true martingales. In the finite dimensional case (see \cite{CKT:12}), but also in the infinite dimensional case when dealing with compact state spaces  (see the probability measure case in \cite{CLS:18}), this is always true for \emph{all} polynomials in $P^D$, see Remark \ref{rem7} below.

Since in our setting this does necessarily hold true,  we need to include the true martingale property as an additional assumption (see condition in Theorem \ref{thm1}~\ref{itiinew} below). In  Section~\ref{sec33} we then illustrate some conditions under which this assumption is satisfied.  This is in particular the case if $D=B^*$.
Before stating the theorem, define 
\begin{equation}\label{eqn21}
p_{\vec\A}(y):=(\vec \A \cdot \overline{y})\qquad\text{and}\qquad Lp_{\vec\A}(y):=(L_k \vec \A \cdot \overline{y})
\end{equation}
for all $\vec\A\in\Dcal(L_k)$ and $k\in \N_0$. 
As in finite dimensions the moment formula corresponds to a solution of a system of linear ODEs. In the current infinite dimensional setting we need to make the solution concept precise.

\begin{definition}\label{def:sol}
Let $\mathcal{B}$ be a subset of $ \bigoplus_{j=0}^k (\B^{\otimes j})^{**}$.
We call  a function $t \mapsto \vec \A_t$ 
with values in $\mathcal{D}(L_k)$  a  $\mathcal{B}$-solution of the $k+1$ dimensional system of ODEs
\[
\partial_t \vec\A_t =  L_k \vec\A_t, \quad \vec\A_0=\vec\aaa,		
\]
if for every $t >0$ it holds
$
(\vec \A_t \cdot \vec {y}) = (\vec \aaa \cdot \vec{y})+\int_0^t  (L_k \vec a_s \cdot \vec{ y})ds
$
    for all $\vec\y \in \mathcal{B}$.
\end{definition}

%\begin{remark}
This solution concept resembles at first sight weak solutions due to the pairing with $\vec\y\in \mathcal{B}$. We however require here that $\vec \A_t \in \mathcal{D}(L_k)$ which corresponds rather to a strong solution. The pairing with $\vec \y$ allows in particular to 
avoid Bochner integration.
%\end{remark}

We are now ready to state the main result of this section. To simplify the notation set $\overline \Scal_k:=\{\overline \y=(1,\y,\ldots,\y^{\otimes k})\colon \y\in \Scal\}$.

\begin{theorem}\label{thm1}
Let $L:P^D\to P$ be a polynomial operator, fix a $k$-th dual operator $L_k$, and assume that $L_k$ is closable with domain $\Dcal(L_k)$.
Let $(\lambda_t)_{t \geq 0} $ be an $\Scal$-valued  polynomial process corresponding to $L$,
and fix $\vec\aaa=(a_0, \ldots, a_k)\in\Dcal(L_k)$. Suppose that the following conditions hold true.
\begin{enumerate}
\item \label{iti}
There is a $\overline \Scal_k$-solution in the sense of Defintion \ref{def:sol}  of the $k+1$ dimensional system  of linear ODEs on $[0,T]$ given by
\begin{equation}\label{eqn12}
\partial_t \vec\A_t =  L_k \vec\A_t, \qquad 
\vec\A_0=\vec\aaa.					
\end{equation}
\item\label{itiinew} The process $(N_t^{p_{\vec\A_s}})_{t\in[0,T]}$ given by \eqref{IIeqnN} for $ p_{\vec\A_s}$ and $Lp_{\vec\A_s}$ as in \eqref{eqn21},
 defines a true martingale for each $s\in[0,T]$.
 \item\label{itii}
$\int_0^{T}\int_0^{T}\big|\E[
Lp_{\vec\A_{s}} (\lambda_{u})]\big|dsdu<\infty$.
\end{enumerate}
Then the following conditional moment formula holds true for all $0\leq t\leq T$.
\begin{equation*}
\mathbb{E}[   \aaa_0+ \langle \aaa_1, \lambda_T \rangle+\ldots+ \langle \aaa_k, \lambda_T^{\otimes k} \rangle \,|\, \mathcal{F}_t ]=\A_{T-t,0}+\langle \A_{T-t,1}, \lambda_t\rangle+\ldots+ \langle \A_{T-t,k}, \lambda^{\otimes k}_t\rangle,
\end{equation*}
i.e. in short hand notation
$
\mathbb{E}[ (\vec \A_0 \cdot \overline{\lambda}_T) \,|\, \mathcal{F}_t ]=( \vec \A_{T-t } \cdot \overline{\lambda_t}).
$
\end{theorem}

\begin{proof}
We will follow the proof of Theorem 4.4.11 in \cite{EK:05} in order to obtain a slightly more general result (compare also with \cite{CLS:18}).

Fix $T\in \R_+$, $t\in [0,T]$ and $A\in \Fcal_t$. 
For all $(s,u)\in[0,T-t]\times [0,T-t]$ define 
$$f(s, u) :=\E[(\vec\A_{s}\fdot \overline{\lambda}_{t+u})1_A].$$
Fix $u\in[0,T-t]$ and note that equation \eqref{eqn12} yields
\begin{align*}
f(\overline s, u) -f(\underline s, u)
=\E[\left(
  (\vec\A_{\overline s}\cdot \overline{\lambda}_{t+u} )-( \vec\A_{\underline s}\cdot {\overline{\lambda}_{t+u}} ) \right)  1_A]
  =\int_{\underline s}^{\overline s} \E[
 (   L_k\vec\A_s\cdot {\overline{\lambda}_{t+u}} ) 1_A]ds,
\end{align*}
for all $\overline s, \underline s,\in[0,T-t]$.
Fix then $s\in[0,T-t]$ and note that condition~\ref{itiinew} yields
$$
f(s, \overline s) -f(s, \underline s)
=\E[\E[
  ( \vec\A_{s}\cdot {\overline{\lambda}_{t+\overline s}})  -   (\vec\A_{s}\cdot {\overline{\lambda}_{t+\underline s}})  
|\Fcal_t] 1_A]=
\int_{\underline s}^{\overline s} \E[
(    L_k\vec\A_{s}\cdot \overline{\lambda}_{t+u}) 1_A]du.
$$
Since $\int_0^{T-t}\int_0^{T-t}\big|\E[
( L_k\vec\A_{s}\fdot { \overline{\lambda}_{t+u}})]\big|dsdu<\infty$ by condition~\ref{itii}, Lemma 4.4.10 in \cite{EK:05} then yields 
\begin{align*}
 \E[(\vec\A_{T-t}\fdot \overline{\lambda}_{t})1_A]&-\E[(\vec\A_{0}\fdot \overline{\lambda}_{T})1_A]
 =f(T-t,0)-f(0,T-t)\\
  &=\int_{0}^{T-t}  \E[(   L_k\vec\A_{s}\cdot { \overline{\lambda}_{T-t -s}} ) 1_A]
  -\E[
(    L_k\vec\A_{s}\cdot  \overline{\lambda}_{T-t-s})1_A]ds
=0
\end{align*}
 and the result follows.
\end{proof}

\begin{example}\label{ex3}
Consider again the setting of Example~\ref{ex2} and  let $(\lambda_t)_{t \geq 0} $  be a \emph{Jacobi diffusion with vanishing drift}, i.e.~a continuous $\R$-valued polynomial process corresponding to $L$, and fix $\vec\aaa\in\R^{k+1}$.
We illustrate now how Theorem~\ref{thm1} can be applied in this setting.
Observe that $L_k$ is well defined on $\Dcal(L_k)=\R^{k+1}$ by definition. 
Observe that the system of linear ODEs given in \ref{iti} of Theorem~\ref{thm1} is given by
$$
\partial_t \vec\A_t =  G_k\vec \A_t, \quad 
\vec\A_0=\vec\aaa					
$$
and it is solved by $\vec\A_t=e^{tG_k}\vec\aaa$, which lies in $\R^{k+1}$. 
As we will see in Remark~\ref{rem7} and Example~\ref{ex11}, Condition~\ref{itiinew} of Theorem~\ref{thm1} is always satisfied in the finite dimensional case. Since $L_k$ is a bounded operator, continuity of  $(\vec \A_t)_{t\geq0}$ is enough to guarantee Condition~\ref{itii} of Theorem~\ref{thm1}. 
We can thus conclude that
\begin{align*}
\mathbb{E}[\aaa_0+  \aaa_1\lambda_T +\ldots+  \aaa_k \lambda_T^k \,|\, \mathcal{F}_t ]&=\A_{T-t,0}+\ \A_{T-t,1}\lambda_t+\ldots+  \A_{T-t,k} \lambda^{k}_t\\
&=(1,\lambda_t,\ldots,\lambda_t^k)^\top e^{(T-t)G_k}\vec\aaa.
\end{align*}

\end{example}

\subsection{Bidual moment formula}

Let us now pass to the bidual moment formula, which involves the bidual operator $M_k$.
Before stating the result we need to introduce a notion of integration for $\B^{\otimes k}$-valued maps. One possibility is to use the notion of the Dunford integral (see, e.g.~\cite{R:13}).

\begin{definition}
 Let $\lambda:\Omega\to\B^{\otimes k}$ be weakly-measurable in the sense of Definition~\ref{def3}.
 We say that $\lambda$ is Dunford integrable if
$$
\E[|\langle \A,\lambda\rangle|]<\infty\qquad\text{and}\qquad\E[\langle \A,\lambda\rangle]=\langle \A,m\rangle
$$
for some $m\in (\B^{\otimes k})^{**}$ and for all $\A\in (\B^{\otimes k})^*$. In this case we write $\E[\lambda]=m$. 
\end{definition}

Observe that the identity $\E[\langle \A_k,\lambda\rangle]=\langle \A_k,m\rangle$ for all $\A_k\in  (\B^{\otimes k})^*$ follows directly from the definition.

For the bidual moment formula, we shall need the following weak solution concept.

\begin{definition}\label{def:solweak}
Fix a $k$-th dual operator $L_k$ and a $k$-th bidual operator $\M_k$.
Let $\mathcal{H} \subseteq \mathcal{D}(L_k)$. 
We call a function $t \mapsto \vec m_t$ with values in $\bigoplus_{j=0}^k (\B^{\otimes j})^{**}$  a ${\mathcal{H}}$-weak  solution of the $k+1$ dimensional system of ODEs
\[
\partial_t \vec m_t =  M_k \vec m_t, \quad \vec m_0=\vec m,	
\]
if for every $t >0$ and $a \in \mathcal{H}$ it holds
$
(\vec a \cdot \vec m_t) = (\vec a \cdot \vec m)+\int_0^t  (L_k \vec a \cdot \vec m_s)ds.
$
\end{definition}

%\begin{remark}
Note that, in contrast to Definition \ref{def:sol}, we here deal with a truly weak solution concept since the adjoint operator $L_k$ is involved.
%\end{remark}

Fix now a dual operator $L_k$ as in Theorem \ref{dual2} below and set - 
using the notation of \eqref{IIeqnN} and \eqref{eqn21} -
$$\mathcal{E}:=\{\vec\A\in\Dcal(L_k)\colon N^{p_{\vec\A}}\text{ is a true martingale}\}.$$
To ease notation, we here do not indicate the dependence on $L_k$.

%%%DUAL2
\begin{theorem}\label{dual2}
Let $L:P^D\to P$ be a polynomial operator, fix a $k$-th dual operator $L_k$and a $k$-th bidual operator $\M_k$, and assume that $L_k$ is closable with domain $\mathcal{D}(L_k)$. Let $(\lambda_t)_{t \geq 0} $ be an $\Scal$-valued polynomial process corresponding to $L$. Suppose that \begin{equation}\label{iti2} 
\text{$\lambda_t^{\otimes j}$ is Dunford integrable for all $j\in\{0,\ldots, k\}$ and $t >0$,}
\end{equation}
and set
$\vec m_{t}:=(1,\E[\lambda_t],\ldots,\E[\lambda_t^{\otimes k}]).$
Then
$(\vec m_t)_{t\geq0}$
is a $\Ecal$-weak solution of 
the $k+1$ dimensional system of linear ODEs given by
\begin{align}\label{momentODE}
\partial_t  \vec m_t =  \M_k\vec m_t, \qquad 
\vec m_0=(1,\lambda_0,\ldots,\lambda_0^{\otimes k}).					
\end{align}
\end{theorem}

\begin{remark}
Condition \eqref{iti2} explicitly reads as,
$$
\sup\{\E[\langle \A_j,\lambda_t^{\otimes j}\rangle]\ :\ \A_j\in (\B^{\otimes j})^*,\ \|\A_j\|_{*j}\leq1\}<\infty,\qquad j\in\{0,\ldots, k\}.
$$
 Since for $k$ even, the map $\y^{\otimes k}\mapsto\|\y\|^k_\times$ can be extended to an element of $(\B^{\otimes k})^*$, this condition is equivalent to
$\E[\|\lambda_t\|^k]<\infty$. It is thus automatically satisfied if condition \eqref{eqn16} below is in force.

\end{remark}

\begin{proof}
Fix $\vec\A\in\mathcal{E}$ and set $p_{\vec\A}$ and $Lp_{\vec\A}$ as in \eqref{eqn21}. Recall that by definition of $\mathcal{E}$ we have that $N^{p_{\vec\A}}$  is a true martingale and thus
\begin{equation}\label{eqn10}
\E[p_{\vec\A}(\lambda_t)] - p_{\vec\A}(\lambda_0) - \int_0^t \E[Lp_{\vec\A}(\lambda_s)] ds =0.
\end{equation}
Recall that $\E[(\vec \A \cdot \overline{\lambda_t})]=(\vec \A \cdot \vec {m_t})$ and 
$(\vec \A \cdot \overline{\lambda_0})=(\vec \A \cdot \vec m_0)$. 
Moreover, by the defintion of $L_k$ we have
$
\E[Lp_{\vec\A}(\lambda_s)] =\E[(L_k\vec\A \cdot \overline {\lambda}_s)]=(L_k\vec\A \cdot \vec {m}_s). 
$
  Plugging those terms in \eqref{eqn10} yields
$
(\vec a \cdot \vec m_t) = (\vec a \cdot \vec m_0)+\int_0^t  (L_k \vec a \cdot \vec m_s)ds
$
and thus the assertion.
\end{proof}

\begin{remark}\label{rem:motivation}
We are now in the position to explain why we decided to work with the potentially very large bidual space $(B^{\otimes j})^{**}$ instead of the space $B^{\otimes j}$ itself. 
As we will see in the applications (see for instance Section~\ref{sec43}), typically the solution $(\vec m_t)_{t\geq0}$ of \eqref{momentODE}  does not   belong to $B^{\otimes j}$. A possible alternative choice would be to work with the closure $\overline{B^{\otimes j}}$ of $B^{\otimes j}$ with respect to some cross norm. This would however have two main disadvantages: first, elements of $\overline{B^{\otimes j}}\setminus B^{\otimes j}$ are just abstractly defined and checking if some $\y_j\in \overline{B^{\otimes j}}$ is typically quite involved. Second, every element of $\overline{B^{\otimes j}}$  corresponds to an element of  $({B^{\otimes j}})^{**}$, which implies that requiring that $\y_j\in ({B^{\otimes j}})^{**}$ is less restrictive than requiring that $\y_j\in\overline{B^{\otimes j}}$.
\end{remark}

Observe that Theorem~\ref{dual2} does not guarantee that a solution of the given system of ODEs coincides with the deterministic process $(1,\E[\lambda_t],\ldots,\E[\lambda_t^{\otimes k}])$. 
Indeed, this result can fail if the solution of \eqref{momentODE} is not unique.
Let us here state the precise notion of uniqueness that is needed.

\begin{corollary}\phantomsection\label{cor11}
Suppose that given two $\Ecal$-weak solutions $(\vec m_t^1)_{t \geq 0}$,  $(\vec m_t^2)_{t \geq 0}$ of \eqref{momentODE} we have $\vec m_t^1= \vec m_t^2$ for all $t \geq 0$. Then 
$
\vec m_t^1=\vec m_t^2= (1,\E[\lambda_t],\ldots,\E[\lambda_t^{\otimes k}]).
$

Fix $\Hcal  \subseteq \Ecal$.
Suppose that  given two $\Ecal$-weak solutions $(\vec m_t^1)_{t \geq 0}$,  $(\vec m_t^2)_{t \geq 0}$ of \eqref{momentODE} we have $(\vec a  \cdot \vec m_t^1)= (\vec a  \cdot \vec m_t^2)$ for all $\vec a \in \Hcal$ and  $t \geq 0$. Then 
$
(\vec a \cdot \vec m^1_t)=(\vec a \cdot \vec m^2_t)=\E[(\vec a \cdot \overline{\lambda}_t)],
$
for each $\vec a \in \Hcal$.
\end{corollary}

\begin{remark}
Suppose that the martingale problem for $L$ is well-posed implying that the corresponding semigroup $(P_t)_{t \geq 0}$ is well-defined on $P^D$. Assume that it can be uniquely extended to $P$. Then $P_tp_{\vec a}(\lambda)=E[p_{\vec a}(\lambda_t)]$ uniquely solves the abstract Cauchy problem given by
\begin{equation}\label{eq:Cauchy}
\begin{aligned}
\partial_t u(t,\lambda)&= \overline L u(t, \lambda),\quad t \geq 0\\
u(0,\lambda)&=p_{\vec a}(\lambda)= (\vec a \cdot \overline{\lambda}),
\end{aligned}
\end{equation}
where $\overline L$ denotes the extension of $L$ as generator of $(P_t)_{t \geq 0}$.
Let now $\vec m_t$ be a $\Ecal$-weak solution of \eqref{momentODE} with $\vec m_0=\overline \lambda$. Then $\Ecal$-weak uniqueness holds, as
 $(\vec a \cdot \vec m_t)$ solves \eqref{eq:Cauchy} and by uniqueness of the solution to the Cauchy problem this has to be equal to  $P_tp_{\vec a}(\lambda)$.
\end{remark}

The next corollary provides other sufficient conditions under which any solution $(m_t)_{t \geq 0}$  of \eqref{momentODE} satisfies
$
\mathbb{E}[ (\vec \A \cdot \overline{\lambda}_T)]= (\vec \aaa \cdot \vec m_T)$ for $ \vec \A \in \mathcal{D}(L_k).$
Recall the notation $\overline \Scal_k:=\{\overline \y=(1,\y,\ldots,\y^{\otimes k})\colon \y\in \Scal\}$.
%%%COROLLARY
\begin{corollary}\label{cor1}
Let $L:P^D\to P$ be a polynomial operator, fix a closable $k$-th dual operator $L_k$ with domain $\Dcal(L_k)$ and a $k$-th bidual operator $\M_k$. 
Let $(\lambda_t)_{t \geq 0} $ be an $\Scal$-valued polynomial process corresponding to $L$ and fix $\vec\aaa\in\Dcal(L_k)$. 
Suppose that the following conditions hold.
\begin{enumerate}
\item\label{iti3} Let $(\vec\A_t)_{t\geq0}$ be an $\overline \Scal_k$-solution of \eqref{eqn12} such that the conditions 
of Theorem~\ref{thm1} are satisfied and denote by $\Rcal$ the  set $\Rcal:=\{\vec\A_t\colon t\geq 0\}$.

\item\label{itii3} 
There is a $\Rcal$-weak solution in the sense of Definition \ref{def:solweak} of
the $k+1$ dimensional system of linear ODEs given by \eqref{momentODE}.
\item\label{itii3b} $(\vec\A_t)_{t\geq0}$ can be  paired with $(\vec m_s)_{s \geq 0}$, that is
$(\vec\A_t)_{t\geq0}$ satisfies additionally for all $t,s \geq 0$
\[
(\vec \A_t \cdot \vec m_s) = (\vec \aaa \cdot \vec m_s)+\int_0^t  (L_k \vec a_u \cdot \vec m_s)du.
\]
\item\label{itiii3} $ \int_0^T\int_0^T |(L_k \vec \A_s \cdot \vec m_t)| ds dt< \infty$.
\end{enumerate}
Then $\mathbb{E}[ (\vec \aaa \cdot \overline{\lambda}_T)]= (\vec \aaa \cdot \vec m_T)$ holds for all $T\geq0$.

\end{corollary}

\begin{proof}
The proof consists in proving $(\vec \A_{T} \cdot \vec  m_{0})=(\vec \aaa \cdot \vec m_T)$ and then applying the dual moment formula.
Set 
$f(s,t):=(\vec \A_s\cdot \vec m_t)$
and $F(s,t):=( L_k \vec\A_{s} \cdot \vec m_{t})$ so that  due to the  Conditions \ref{iti3}, \ref{itii3}, and \ref{itii3b} we have
$f(\overline s, t) -f(\underline s, t)=\int_{\underline s}^{\overline s}F(s,t) ds
$ and $
f( s, \overline t) -f( s, \underline t)=\int_{\underline t}^{\overline t}F(s,t) dt.
$
This together with Condition \ref{itiii3} and Lemma 4.4.10 in \cite{EK:05} then yields $f(T,0)= (\vec \A_{T} \cdot \vec  m_{0})=(\vec \A \cdot \vec m_T)=f(0,T)$.
Since $\vec m_0=(1,\lambda_0,\ldots,\lambda_0^{\otimes k})$ and $\mathbb{E}[ \vec \A \cdot \overline{\lambda}_T]= (\vec \A_T \cdot \overline{\lambda}_0)$ by Theorem~\ref{thm1}, the claim follows.
\end{proof}

\begin{remark}
Consider the setting of Corollary~\ref{cor1} and let $(\vec\A_t)_{t\geq0}$
be given by \ref{iti3}. A deterministic process $(\vec m_t)_{t\geq0} $ 
then satisfies Conditions \ref{itii3} and \ref{itii3b} if and only if $ (\vec \A_{s} \cdot \vec m_{t} )$ is absolutely continuous in $s$ and $t$ and satisfies
\begin{align*}
\partial_t ( \vec \A_{s}\cdot \vec m_{t})& = (L_k \A_{s} \cdot \vec  m_{t} ), & &
\vec m_0=(1,\lambda_0,\ldots,\lambda_0^{\otimes k}),\\
\partial_s ( \vec \A_{s} \cdot \vec m_{t})& = ( L_k\vec\A_{s} \cdot \vec m_{t} ), & &
\vec \A_0=\vec\A.
\end{align*}
\end{remark}

\begin{example}\label{ex4}
Consider again the setting of Example~\ref{ex3}. We now  illustrate how Theorem~\ref{dual2} and Corollary~\ref{cor1} can be applied in this setting.
Since $\E[\lambda_t^j] < \infty $  for each $j$,  we can set $\vec m_t:=(1,\E[\lambda_t],\ldots,\E[\lambda_t^k])$ and 
 the conditions of Theorem~\ref{dual2} are satisfied. We thus get that
$\vec m_t$ 
is a solution of the system of linear ODEs given by
$
\partial_t  \vec m_t =  \M_k\vec m_t, $ for $ 
\vec m_0=(1,\lambda_0,\ldots,\lambda_0^k).					
$
Since also the conditions of Corollary~\ref{cor1} are satisfied (or more directly, since this system has a unique solution)
 we can conclude that
 $$(1,\E[\lambda_t],\ldots,\E[\lambda_t^k])^\top=e^{tG_k^\top}(1,\lambda_0,\ldots,\lambda_0^k)^\top.$$
\end{example}

%%%Remark

\begin{remark}\label{rem8}

An inspection of the moment formulas gives the impression that the dual moment formula is more suitable for computing $\E[p(\lambda_t)]$ for some fixed polynomial $p \in P^D$. In practice, it can however happen that the system of linear ODEs given by \eqref{eqn12}
is harder to solve than its adjoint given by \eqref{momentODE}.
The application presented in Section~\ref{sec43} is a clear instance of this situation.  If this is the case, the bidual moment formula can be used to provide an heuristic ansatz for a solution $(\vec \A_t)_{t\geq0}$ of the dual system of ODEs. Indeed, let $(\vec m_t)_{t\geq0}$ be a solution of the bidual ODE system and assume that both moment formulas hold. Then the relation
$(\vec \A_t\cdot\overline \lambda_0)=\E[(\vec \A\cdot\overline \lambda_t)|\lambda_0]=(\vec \A\cdot \vec m_t),$
holds and can be used as defining property for $(\vec \A_t)_{t\geq0}$.
\end{remark}

\subsection{Some pratical conditions for applying Theorem~\ref{thm1}}\label{sec33}

We here provide some sufficient conditions which imply  \ref{itiinew} and \ref{itii} of Theorem~\ref{thm1}.
Throughout the section we  let $L:P^D\to P$ be a polynomial operator and $L_k$   
a closable $k$-th dual operator with domain $\Dcal(L_k)$.
We also 
let $(\lambda_t)_{t \geq 0} $ be a polynomial process corresponding to $L$ and we assume that $(\vec\A_t)_{t\geq0}$ is a $\overline \Scal_k$-solution of   \eqref{eqn12} in the sense of Definition~\ref{def:sol}.
A first intuitive sufficient condition that implies Conditions \ref{itiinew} and \ref{itii} of Theorem~\ref{thm1} is provided in the following lemma.
\begin{lemma}\label{lem:sufficient}
If $\E[\sup_{t\leq T}\|\lambda_t\|^{k}]<\infty$ for all $k\in \N$ , then condition \ref{itiinew} of Theorem~\ref{thm1} is satisfied. In this case, condition \ref{itii} of the same theorem is implied by
\begin{equation}\label{eqn104}
\int_0^T\|  {L_k^j}\vec\A_{s}\|_{*j}ds<\infty,\qquad\text{for all } j\in\{0,\ldots,k\}.
\end{equation}
\end{lemma}
\begin{proof}
Set $\Lambda_k:=(1+\sup_{t\in[0,T]}\|\lambda_t\|^{k})$ and set $p_{\vec\A}$ and $Lp_{\vec\A}$ as in \eqref{eqn21}. Since  for each $\vec \A\in  \Dcal(L_k)$ it holds
\begin{equation}\label{eqn22}
\sup_{t\in[0,T]}|p_{\vec \A}(\lambda_t)|\leq 
\Lambda_k\sum_{j=0}^k\| \A_{j}\|_{*j}
\quad\text{and}\quad
\sup_{t\in[0,T]}| Lp_{\vec \A}(\lambda_t)|\leq 
\Lambda_k\sum_{j=0}^k\|  {L_k^j}\vec\A\|_{*j},
\end{equation}
we can see that $\E[\sup_{t\leq T}\|\lambda_t\|^{k}]<\infty$ guarantees that the process $N^p$ given by \eqref{IIeqnN} is a true martingale for each $p\in P^D$.
The same bound together with the dominated convergence theorem can be used to prove that $N^{p_{\vec\A}}$ is a true martingale for each $\vec\A\in \Dcal(L_k)$. The second part of the statement follows from \eqref{eqn22}.
\end{proof}

We now move to a different condition which in particular guarantees that condition \ref{itiinew} is always satisfied in the classical cases.
\begin{definition}\phantomsection\label{def2}
\begin{enumerate}
\item 
We say that $p\in P^D$ is $(C,q)$-bounded on $\Scal$ if there exists a constant $C >0$ and a polynomial $q\in P^D$ such that on $\Scal$ we have
$$p^2\leq C q,\qquad  (Lp)^2\leq C q,\qquad\text{and}\qquad |L q|\leq C q.$$

\item For general $p \in P$,
we also call it $(C,q)$-bounded on $\Scal$, if 
$(p,Lp)$ can be approximated by a sequence $((p_n,Lp_n))_{n\in \N}$ with $p_n\in P^D$ being $(C,q)$-bounded in the sense of (i).
\end{enumerate}
\end{definition}

The reason why the property defined in Definition~\ref{def2} is so important can be seen from the next lemma.
\begin{lemma}\label{lem8}
Fix   $\vec\A\in \Dcal(L_k)$, set $p_{\vec\A}$ and $Lp_{\vec\A}$ as in \eqref{eqn21}, and suppose that $p_{\vec\A}$ is $(C,q)$-bounded for some $C >0$ and $q\in P^D$ with $q(\lambda_0)=1$.
Then for all $t\geq0$
$$\E[q(\lambda_t)]\leq e^{ C t},\qquad\E[p_{\vec\A}(\lambda_t)^2]\leq Ce^{ C t}, \qquad \E[(Lp_{\vec\A}(\lambda_t))^2]\leq Ce^{ C t}, \qquad  \E[\sup_{s\leq t}(N_s^{p_{\vec\A}})^2]<\infty.$$
In this case,
the process $N^{p_{\vec\A}}$ is a square integrable martingale.
\end{lemma}
\begin{proof}
For $p_{\vec\A}\in P^D$ the proof follows the proof of Theorem~2.10 in \cite{CKT:12}, where  the stopping times are chosen to be localizing sequences for $N^{q}$ and $N^{p_{\vec\A}}$ and the role of $F$ there is now taken by $q$. For $p_{\vec\A}\notin P^D$ the claim follows by the dominated convergence theorem.  Indeed, 
by \eqref{IIeqnN} $(C,q)$-boundedness yields
$$\E[(N^{p}_t)^2]=\E[\lim_{n\to\infty}(N^{p_n}_t)^2]
\leq 3C\E\bigg[q(\lambda_t)+q(\lambda_0)+t \int_0^t q(\lambda_u)du\bigg]<\infty,$$
    where in the last inequality we use that $q(\lambda_t)$ is integrable due to the first part of the proof.
Analogously, since 
$|N^{p_n}_t-N^{p_n}_s|
\leq 2C(q(\lambda_t)+q(\lambda_s)+(t-s) \int_0^t q(\lambda_u)du),$
 the dominated convergence theorem yields
$\E[(N^{p}_t-N^{p}_s)1_A]
%=\E[\lim_{n\to\infty}(N^{p_n}_t-N^{p_n}_s)1_A]
=\lim_{n\to\infty}\E[(N^{p_n}_t-N^{p_n}_s)1_A]=0,$
for all $A\in \Fcal_s$, proving the martingale property.
\end{proof}

\begin{remark}\label{rem7}
Note that in the above definition the constants $(C,q)$ can always depend on the polynomial $p$.
In classical cases (see Example~\ref{ex11} below) we can however choose $q$ uniformly for all polynomials up to degree $m$, say. Indeed,
there typically exists a polynomial $q_m\in P^D$ such that each $p\in P^D$ with $\deg(p)\leq m$ is $(C_p,q_m)$-bounded  for some constant $C_p$  that can depend on the polynomial $p$ (that is why we make this dependence now explicit). More precisely, in such cases there is a $q_m\in P^D$ satisfying 
\begin{equation}\label{eqn16}
1+\|\y\|^{2m}\leq K  q_m(\y)\qquad\text{for all}\qquad\y\in \Scal,
\end{equation}
for some constant $K$. Since each $p\in P$ with $\deg(p)\leq m$ satisfies $p(\y)^2\leq K_p(1+ \|\y\|^{2m})$ for some constant $K_p$, condition \eqref{eqn16}  implies that $p_{\vec\A}$  is $(C_p,q_m)$-bounded for each $\vec\A\in \Dcal(L_k)$, where $C_p=(K_p+K_{Lp}+K_{q_m})K$.

Observe that condition~\eqref{eqn16} has two important consequences. First, it guarantees that $N^{p_{\vec\A}}$ is a square integrable martingale for each $\vec\A\in \Dcal(L_k)$, and thus
Condition~\ref{itiinew} is satisfied. Second, 
since
$$
p_{\vec\A_{s}}(\y)=\sum_{j=0}^k\langle  {L_k^j}\vec\A_{s},\y^{\otimes j}\rangle
\leq (1+\|\y\|^{2k})\sum_{j=0}^k\|  {L_k^j}\vec\A_{s}\|_{*j}
\leq Kq_k(\y)\sum_{j=0}^k\|  {L_k^j}\vec\A_{s}\|_{*j}
$$
one can see that Condition \ref{itii} in Theorem \ref{thm1} is implied by \eqref{eqn104}.
\end{remark}

\begin{example}\label{ex11}
As mentioned before condition \eqref{eqn16} is always satisfied in the classical cases. 
Examples include the case where $\B$ is finite dimensional ($q_m(\y):=1+\sum_i \y_i^{2m}$), the case where $\Scal$ is the set of finite positive measures on some underlying space $E$ ($q_m(\y)=1+\y(E)^{2m}$, where $y(E)$ denotes the total mass of $\y$), and the case where  $\Scal$ is bounded ($q_m(y)=1$).
\end{example}

Observe that if condition \eqref{eqn16} does not hold we cannot expect $N^p$ to be a true martingale for each $p\in P^D$. Indeed, $(C,q)$-boundedness on $\Scal$ does not need to hold for each $p\in P^D$.

\begin{lemma}\label{lem14}
Let $(\vec\A_t)_{t\geq0}$ satisfy Condition  \ref{iti} of Theorem~\ref{thm1}.
If $p_{\vec\A_s}$ is $(C,q_{s})$-bounded  for each $s\in[0,T]$ and some family of polynomials $(q_s)_{s \in [0,T]}$, then Conditions~\ref{itiinew} and \ref{itii} of Theorem~\ref{thm1} are satisfied.
\end{lemma}

\begin{proof}
Condition \ref{itiinew} follows from Lemma~\ref{lem8}. By the same lemma we can also 
 compute
$$
\sup_{s,u\in[0,T]} \big|\E[
Lp_{\vec\A_{s}}(\lambda_{u})]\big|
\leq  \sup_{u\in[0,T]}
(1+Ce^{ Cu})
<\infty,
$$
proving that condition \ref{itii} is satisfied as well.
\end{proof}

\section{Examples and applications}\label{sec:Examples}

This section is devoted to show the connection of our general setup with some examples from the literature as well as 
the wide applicability of the previously derived moment formulas, e.g.~for forward variance modeling and the computation of the expected signature.

\subsection{Generic polynomial operators}

We start by introducing generic polynomial operators of L\'evy  type (see also Section~4 in \cite{LS:19}).
To do so, we briefly recall the notion of the Fr\'echet derivative.
\begin{definition} Let $(\B,\|\fdot\|)$ be a Banach space.
A map $f:\B \to\R$ is said to be Fr\'echet differentiable at $\y\in \B $ if  
$$\lim_{\|\widetilde \y \|\to0}\frac{|f(\y+\widetilde\y)-f(\y)-\langle\partial f(\y),\widetilde\y\rangle|}{\|\widetilde\y\|}=0,$$ 
for some $\partial f(\y)\in \B^*$.
Analogously, whenever it exists, we denote by $\partial^k f(\y)$ the element of $(\B^{\otimes k})^*$ corresponding to the $k$-th iterated Fr\'echet derivative of $f$ at $\y$. 
\end{definition}

Observe in particular that for every sufficiently differentiable $\phi:\R^d\to\R$ and $\A_1,\ldots,\A_d\in \B^*$ setting $p(\y):=\phi(\langle \A_1,\y\rangle,\ldots,\langle\A_d,\y\rangle)$ we have that $p$ is Fr\'echet differentiable at each $\y$ in $\B$ and 
\begin{equation}\label{eqn20}
\partial^n p(\y)=\sum_{i_1,\ldots, i_n=1}^d \phi_{i_1,\dots, i_n}\big(\langle \A_1,\y\rangle,\ldots,\langle\A_d,\y\rangle\big) \A_{i_1}\otimes\ldots\otimes \A_{i_n},
\end{equation}
where $\phi_{i_1,\dots, i_n}(x):=\frac{d^n\phi}{dx_{i_1}\cdots dx_{i_n}} (x)$. This in particular implies that for each $p\in P^D$ for some $D\subseteq B^*$ we have that $\partial^n p(\y)\in D^{\otimes n}$ for all $\y\in \B$.
With the notion of the Fr\'echet derivative we can now show how generic polynomial operators $L:P^D\to P$ look like.
\begin{lemma}
 Let $\Scal\subseteq \B $, fix a linear subspace $D\subseteq \B^*$, and let $L: P^D \to P$ be a linear operator. 
Suppose that $L:P^D\to P$ acts on test functions $p\in P^D$ by
\begin{equation*}
Lp(\y) = 
 B(\partial p(\y), \y) + \frac12  Q(\partial^2p(\y), \y)+ \int_{\Scal}(p(\z)-p(\y)-\langle\partial p(\y),\z-\y\rangle) N(\y,d\z),
\end{equation*}
for each $\y\in \Scal$, where 
\begin{itemize}
\item $B(\fdot,\y)$ is a linear operator from $D$ to $\R$ and $B(\A,\y)$ is a polynomial of degree at most $1$ on $\B$ for each $\A\in D$.
\item $N(\y,d\z)$ is a measure on $\Scal$ such that $\int_{\Scal}\langle\A,\y-\z\rangle^k N(\y,d\z)$ is a polynomial of degree at most $k$ on $B$ for each  for each $\A\in D$ and each $k\in\{3,4,\ldots\}$.
\item $Q(\fdot,\y)$ is a linear operator from $D\otimes D$ to $\R$, $Q(\A\otimes \A,\y)\geq 0$,  and $Q(\A\otimes \A,\y)+\int_{\Scal}\langle\A,\y-\z\rangle^2 N(\y,d\z)$ is a polynomial of degree at most $2$ on $\B$ for each $\A\in D$.
\end{itemize}
Then $L$ is $\Scal$-polynomial. Moreover, the 
drift, diffusion, and jump behavior of the corresponding $\Scal$-valued polynomial process $(\lambda_t)_{t\geq0}$ is governed by these objects, meaning that for each $\A_1,\ldots,\A_d\in D$, 
the $\R^d$-valued process  $((\langle\A_1,\lambda_t\rangle,\ldots,\langle \A_d,\lambda_t\rangle))_{t\geq0}$ is a semimartingale whose characteristics  $(B^{\vec\A},\widetilde{C}^{\vec\A},\nu^{\vec\A})$ (with $\widetilde{C}^{\vec\A}$ denoting the modified second characteristic in the sense of \cite{JS:03})
satisfy
\begin{align*}
B^{\vec\A}_{t,i}&=\int_0^t B(\A_i, \lambda_s)ds,\\
\widetilde C^{\vec\A}_{t,ij}&=\int_0^t \Big(Q(\A_i\otimes \A_j,\lambda_s)+\int_{\Scal}\langle\A_i,\lambda_s-\z\rangle\langle\A_j,\lambda_s-\z\rangle N(\lambda_s,d\z)\Big)ds\\
\int\xi_1^{k_1}\cdots\xi_d^{k_d}\ \nu_t^{\vec\A}(dt, d\xi)&=dt\int_\Scal\langle\A_1,\z-\lambda_t\rangle^{k_1}\cdots\langle\A_d,\z-\lambda_t\rangle^{k_d}\ N(\lambda_t,d\z)
\end{align*}
for each $k_1,\ldots,k_d\in \N_0$ such that $\sum_{j=0}^d k_j\geq 3$.

\end{lemma}

\begin{proof}
Setting $p(\y):=\langle \A,\y\rangle^k$ and noting that 
\begin{align*}
Lp(\y) &= 
 k\langle \A,\y\rangle^{k-1}B(\A, \y) + \frac{k(k-1)}2  \Big(Q(\A\otimes \A, \y)+ \int_{\Scal}\langle\A,\z-\y\rangle^2 N(\y,d\z)\Big)\langle \A,\y\rangle^{k-2}\\
 &+\sum_{\ell=3}^k\binom k \ell \langle \A,\y\rangle^{k-\ell}\int_{\Scal}\langle\A,\z-\y\rangle^\ell N(\y,d\z),
\end{align*}
 the first part of the claim follows. For the second part we apply  It\^o's formula to the process $(p(\langle\A_1,\lambda_t\rangle,\ldots,\langle \A_k,\lambda_t\rangle))_{t\geq0}$ for all polynomials $p$.  Note that this is justified since $((\langle\A_1,\lambda_t\rangle,\ldots,\langle \A_k,\lambda_t\rangle))_{t \geq 0}$ is a semimartingale due to \eqref{IIeqnN}.
 Comparing the obtained representation with \eqref{IIeqnN} inductively over the degree of $p$ yields the result.
\end{proof}

\subsection{Finite dimensional setting}
 We present now how our results appear in the finite dimensional setting. Polynomial processes on a finite dimensional space have been characterized in \cite{CKT:12}, see also \cite{FL:16}. The particular example of the Jacobi diffusion with vanishing drift has already been presented in Examples~\ref{ex1}, \ref{ex2}, \ref{ex3}, and \ref{ex4}.

Let $\B=\R^d$, $\Scal\subseteq \R^d$, $D=\B^*=\R^d$, and recall that $\bigoplus_{j=0}^k(\R^d)^{\otimes j}=\R^{N_k}$ where 
$N_k$ denotes the dimension of the space $P_k$ of polynomials on $\R^d$ up to degree $k$.
For simplicity, assume that $\Scal$ contains an open set and thus there is a one to one correspondence between $P_k$ and the space of polynomials on $\Scal$.
Fix then a polynomial operator $L:P\to P$, let $H:=(h_1,\ldots,h_{N_k})^\top$ be a basis of $P_k$, and
let $G_k\in\R^{N_k\times N_k}$ be the unique matrix such that
 $$L p_{\vec \A}(y)= H(y)^\top G_k \vec \A,\qquad\text{for all }\vec\A\in\R^{N_k}$$
where $p_{\vec \A}(y)= H(y)^\top\vec \A$.
With this notation, the $k$-th dual operator $L_k:\R^{N_k}\to\R^{N_k}$ is given by $L_k \vec \A:=G_k \vec \A$ and 
the map $(\vec \A_t)_{t\geq0}$ with $\vec \A_t=e^{tG_k}\vec \aaa$ solves
the system of linear ODEs given by \eqref{eqn12} for 
$
\vec \A_0=\vec \aaa$.
 Since the solution of this system is given by $\vec \A_t=e^{tG_k}\vec \aaa$, the dual moment formula (Theorem \ref{thm1}) leads to the classical moment formula for finite dimensional  polynomial processes
$$\E[p_{\vec \aaa}(\lambda_T)^\top|\Fcal_t]=H(\lambda_t)^\top e^{(T-t)G_k} \vec\aaa.$$
On the other hand, by \eqref{eqn1} we know that  $M_k\vec \y:=G_k^\top \vec \y$ for all $\y\in \R^{N_k}$.  Since
the map $(\vec m_t)_{t\geq0}$ with $\vec m_t=e^{tG_k^\top}H(\lambda_0)$ is in fact the unique solution (and $\Ecal$-weak-solution) of \eqref{momentODE} for $\vec m_0= H(\lambda_0)^\top$, the bidual moment formula (Theorem~\ref{dual2}) yields 
$$\E[H(\lambda_T)|\lambda_0]= e^{TG_k^\top}H(\lambda_0).$$
This result generalizes to 
$\E[H(\lambda_T)|\Fcal_t]= e^{(T-t)G_k^\top}H(\lambda_t),$ as expected.

%%%PROBABILITIES
\subsection{Probability measure-valued setting}
 Probability measure-valued polynomial diffusions have been studied in \cite{CLS:18}. In that paper, the authors also develop conditions under which existence of solutions of the martingale problem are guaranteed.

Let $(\B,\|\fdot\|)$ be the space of  finite signed measures on a Polish space $E$ and let $\|\fdot\|$ denote the total variation norm. Let $\Scal$ be the space of probability measures on $E$ and $D$ be a dense subset of the space of continuous bounded functions $C_b(E)$ on $E$. Note that in this setting
$$\langle \A,\y\rangle=\int \A(x)\y(dx),\qquad\text{for all }\A\in D,\y\in \Scal.$$
Let then $L:P^D\to P$ be a polynomial operator and observe that for each $k\in\N_0$ there is a $k$-th bidual operator $L_k$ admitting the representation given in Remark~\ref{rem3} for some auxiliary operators  $\Lcal_j:D(\Lcal_j)\to  (\B^{\otimes j})^*$. 
As in \cite{CLS:18}  we additionally assume that $\Lcal_j \A\in C_b(E^j)$ for each $\A\in D(\Lcal_j)$.

Observe  that since $\Scal$ is bounded, Remark~\ref{rem7} and Example~\ref{ex11} yield that condition~\ref{itiinew} of Theorem~\ref{thm1} holds true and condition \ref{itii} of the same theorem is implied by \eqref{eqn104}.   Assume now that
there exists a $C_b(E^k)$-valued (classical) solution $(\A_t)_{t\geq0}$  of the $k$-dimensional  PDEs on $[0,T]$ given by
\begin{equation}\label{eqn18}
\partial_t \A_t( x) =  \Lcal_k (\A_t(\fdot))( x), \quad 
\A_0( x)= \aaa( x),
\end{equation}
satisfying condition \eqref{eqn104}. By Theorem~\ref{thm1} we can then conclude that
\begin{align*}
&\mathbb{E}\Big[  \int \aaa(x_1,\ldots,x_k) \lambda_T(dx_1)\ldots\lambda_T(dx_k)  \,\Big|\, \mathcal{F}_t \Big]
=\mathbb{E}[  \langle \aaa, \lambda_T^{\otimes k} \rangle \,|\, \mathcal{F}_t ]\\
&\qquad= \langle \A_{T-t}, \lambda^{\otimes k}_t\rangle
=\int \A_{T-t}(x_1,\ldots,x_k) \lambda_t(dx_1)\ldots\lambda_t(dx_k),
\end{align*}
for each polynomial process $(\lambda_t)_{t\geq0}$ corresponding to $L$.
This result coincides with the conclusion of Theorem~5.3 in \cite{CLS:18}.

As explained in Remark~5.4 in \cite{CLS:18}, equation \eqref{eqn18} can often be seen as the Kolmogorov backward equation corresponding to an $E^k$-valued process $Z^{(k)}$ with generator $\Lcal_k$. If this is the case the process $(m_{t,k})_{t\geq0}$ given by $m_{t,k}:=\E[\lambda_t^{\otimes k}]$ coincides then with the law of $Z_t^{(k)}$ and the equation given in \eqref{momentODE} (formulated  with $\mathcal{M}_k$ as specified in Remark \ref{rem33}) is given by the Kolmogorov forward equation corresponding to $Z^{(k)}$. We propose now a concrete example (see Example~4.4 in \cite{CLS:18} for more details).

\begin{example}[Fleming-Viot]\label{ex7}
Let  $D=C_0^2(\R)$ be the space of twice continuous differentiable functions on $\R$ vanishing at infinity. The Fleming--Viot diffusion $(\lambda_t)_{t\geq0}$ was introduced by \cite{FV:79} and subsequently studied by several other authors (see e.g.~Chapter 10.4 of \cite{EK:05}). This process takes values  in the space  of probability measures on $\R$, again denoted by $\Scal$.

Recall that $\partial p(\y)\in D$ for all $p\in P^D$ and $\y\in\Scal$ which in particular means that $\partial p(\y)$ is a continuous bounded map on $\R$. We denote by $\partial_x p(\y)$ its evaluation at $x\in \R$. Similarly, $\partial^2 p(\y)\in D^{\otimes 2}$ is a $C_0$-map on $\R^2$ and we denote by $\partial_{x_1x_2}^2 p(\y)$ its evaluation at $x_1,x_2\in \R^2$.
The generator $L$ of a Fleming-Viot diffusion acts on polynomials $p\in P^D$ by
\[
Lp(\y) = \langle \Gcal(\partial p(\y) ), \y\rangle + \frac{1}{2} \langle \Psi(\partial^2 p(\y)),\y^{\otimes 2}\rangle, \quad \y\in \Scal,
\]
where $\Gcal:D\to C_0(\R)$ is given by $\Gcal g:=\frac{1}{2} \sigma^2g''$ for some $\sigma\in \R$ and $\Psi:D\otimes D\to C_0(\R^2)$  by $\Psi\A(x_1,x_2)=\frac 1 2 (\A(x_1,x_1)+\A(x_2,x_2)-2\A(x_1,x_2))$. Observe that $L$ is $\Scal$-polynomial. 

Now, using the representation introduced in Remark~\ref{rem3}, we can see that $\Lcal_1\A(x):=\frac{1}{2} \sigma^2\A''(x)=\Gcal\A(x)$ and
$$
 \Lcal_2\A( x):=\frac 12 \sigma^2\Big(\frac{d}{dx_1^2}\A( x)+\frac{d}{dx_2^2}\A( x)\Big)
 + \int \A( x+\xi)-\A(  x)\Nu( x,d \xi),
$$
where $\Nu( x,d \xi)=\frac 1 2 (\delta_{(0,x_1-x_2)}(d \xi)+\delta_{(x_2-x_1,0)}(d \xi))$. 

One can see that $\Lcal_1$ coincides with the generator of the real valued diffusion $Z^{(1)}$ given by $Z_t^{(1)}=\sigma W_t$ where $(W_t)_{t\geq0}$ denotes a Brownian motion. Moreover, equation \eqref{eqn18}, which reads 
$$\partial_t\A_t(x)=\frac 1 2 \sigma^2\A_t''(x),\qquad \A_0(x)= \aaa(x),$$
coincides with the corresponding Kolmogorov backward equation and  thus with the well known heat equation.
Since it is solved by $(t,x)\mapsto \A_t(x):=\E[ \aaa(Z_t^{(1)})|Z_0^{(1)}=x]$, Theorem~\ref{thm1}  yields
$$\E[\int  \aaa(x)\lambda_T(dx)|\Fcal_t]=\E[ \aaa(Z_{T-t}^{(1)})|Z_0^{(1)}\sim \lambda_t].$$
On the other hand one can see that under the ansatz that $m_{t,1}(dx)=f_t(x)dx$ we have that
$$\int \Lcal_1 \A(x)f_t(x)dx=\int\frac 1 2 \sigma^2 \A''(x)f_t(x)dx=\int\A(x)\frac 1 2 \sigma^2 f''_t(x)dx
,\qquad \text{for all }\A\in C^2_0(\R)$$
showing that the first bidual operator is given by $\Mcal_1(f_t(x)dx)=\frac 1 2 \sigma^2 (f_t''(x)dx)$. By Theorem~\ref{dual2} we can thus conclude that $t\mapsto\E[\lambda_t]:=f_t(x)dx$ satisfies
$$\partial_tf_t(x)=\frac 1 2 \sigma^2f_t''(x),\qquad f_0(x)dx=\lambda_0$$
which, as expected, coincides with the Kolmogorov forward equation for $Z^{(1)}$. Since the conditions of Corollary~\ref{cor1} are satisfied and $f_t(x)dx=\P(Z_t^{(1)}\in\fdot|Z_0^{(1)}=\lambda_0)$ solves the given PDE, we can conclude that
$\E[\lambda_t]=\P(Z_t^{(1)}\in\fdot|Z_0^{(1)}=\lambda_0)$. \\
Let us now focus on $\Lcal_2$. This operator coincides with the generator of a jump-diffusion $(Z_t^{(2)})_{t\geq0}$ taking values in $\mathbb{R}^2$. Between two jumps this process moves like $(\sigma W_t)_{t\geq0}$ where $(W_t)_{t\geq0}$ denotes a 2 dimensional Brownian motion. When a jump occurs, after an exponential time with intensity 1, the process jumps either vertically or horizontally to the diagonal. 
Again, equation \eqref{eqn18} coincides with the Kolmogorov backward equation corresponding to $Z^{(2)}$ and the corresponding solution is given by $(t,x)\mapsto \A_t(x):=\E[ \aaa(Z_t^{(2)})|Z_0^{(2)}=x]$. Theorem~\ref{thm1} then yields
$$\E[\int  \aaa(x_1,x_2)\lambda_T(dx_1)\lambda_T(dx_2)|\Fcal_t]=\E[ \aaa(Z_{T-t}^{(2)})|Z_0^{(2)}\sim \lambda_t\otimes\lambda_t].$$
Proceeding as before we can use Corollary~\ref{cor1} to conclude that the only probability measure $m_{t,2}(dx):=f_t(x)dx$  supported on $\mathbb{R}^2$ satisfying the forward Kolmogorov equation for $Z^{(2)}$ 
$$
\partial_tf_t(x)=\frac 1 2 \sigma^2 \Delta f_t(x)+\frac{1}{2}\int f_t(x)(\delta_{x_2}(dx_1)+\delta_{x_1}(dx_2))-f_t(x))
$$
is given by $m^2_t=\E[\lambda_t\otimes\lambda_t]=\P(Z_t^{(2)}\in\fdot|Z_0^{(2)}\sim \lambda_0\otimes\lambda_0)$.

\end{example}
As a final remark, observe that in the case of Example~\ref{ex7}  the bidual system of ODEs could be formulated using a \emph{relatively strong} formulation. It is however well-known (see for instance \cite{F:08}) that forward Kolmogorov equations can be treated using a weak formulation, consistently with the notion of solution used in Theorem~\ref{dual2}.

\subsection{Polynomial forward variance curve models} \label{sec43}

This section is dedicated to introduce \emph{polynomial forward variance curve models}. The main motivation for this class of models stems from the rather new paradigm of rough volatility (see e.g.,~\cite{ALV:07, GJR:18, BFG:16}). Rough volatility or rough variance is usually introduced via stochastic Volterra processes with singular kernels (e.g., \cite{ALP:17, ACLP:19}).
These processes are non-Markovian, but the Markovian structure can be established by lifting them to infinite dimensions (see \cite{CT:18, CT:19}). One such lift is the forward variance curve, i.e. one considers the curve $x \mapsto \lambda_t(x):=\mathbb{E}[V_{t+x}| \mathcal{F}_t]$ with $(V_t)_{t \geq 0}$ being the (rough) spot variance and $x$ the time to maturity which corresponds to the so-called Musiela parametrization.

Forward variance curve models of course have a longer history and do not just date back to the introduction of rough volatility. Indeed \cite{B:04, B:05, B:08}  proposed them to achieve a market consistent forward skew which cannot be reproduced by traditional stochastic volatility models even with jumps. We refer also to related work by~\cite{B:06}. Instead of modeling the spot volatility or variance, the idea is to specify the dynamics of 
of the forward variance curve, similarly to the Heath-Jarrow-Morton framework in interest rate theory. Due to the martingale property of $(\mathbb{E}[V_T|\mathcal{F}_t])_{t \leq T}$, the dynamics of the forward curve process  $(\lambda_t)_{t \geq 0}$ are necessarily of the form 
\begin{align}\label{eq:polyforward}
d\lambda_t(x)=\Acal\lambda_t(x)dt + dM_t,
\end{align}
for some general function space valued martingale $(M_t)_{t\geq0}$ that  we shall specify as polynomial process.
The $dt$ term of $(\lambda_t)_{t \geq 0}$ is necessarily the first (space) derivative and thus corresponds to the generator of the shift semigroup. In order to make the shift semigroup strongly continuous such that we can treat the above SPDE by standard theory we shall work with the following Hilbert space of forward curves introduced by \cite{F:01}.

Let $\alpha:\R_+\to[1,\infty)$ be a nondecreasing $C^1$-function such that $\alpha^{-1}\in L^1(\R_+)$. Set then
$$\B=\B^*=\{\y\in AC(\R_+,\R) : \|y\|_\alpha<\infty\},$$
 where 
 $AC(\R_+,\R)$ denote the space of absolutely continuous functions from $\R_+$ to $\R$ and 
$\|\y\|_\alpha^2:=|\y(0)|^2+\int_0^\infty|y'(x)|^2\alpha(x)dx$. By Theorem~5.1.1 in \cite{F:01} we know that $\B$ is  a Hilbert space with respect to the scalar product
$$\langle \A,\y\rangle_\alpha:=\A(0)\y(0)+\int_0^\infty \A'(x)\y'(x)\alpha(x)dx.$$
Moreover, by Lemma~3.2 in \cite{BK:14} we also know that $\B\subseteq \R+C_0(\R_+)$, namely the space of continuous functions with continuous continuation to infinity.

In order to simplify the computations we consider here -- instead of norm induced by this scalar product -- the symmetric projective norm
$$\|\y\|_\times:=\inf\Big\{\sum_{i=1}^n|\alpha_i|\|\y_i\|_\alpha^k\colon \y=\sum_{i=1}^n\alpha_i \y_i^{\otimes k}\Big\},\qquad \y\in \B^{\otimes k}.$$
Note that since $\B$ is an Hilbert space, by \cite{F:97} (or also \cite{J:18}) this norm coincides with the projective tensor norm in sense of \cite{R:13} and is thus a crossnorm.
This choice is particularly convenient since in order to check that $\|\A\|_{*k}\leq C$ for some $\A\in (\B^{\otimes k})^*$ it is enough to verify that \begin{equation}\label{eqn8}
|\A(\y^{\otimes k})|\leq C\|\y\|_\alpha^k\qquad\text{ for each }\qquad\y\in \B.
\end{equation}
 Similarly, condition \eqref{eqn8} is enough for checking that a linear map $\A$ belongs to $(\B^{\otimes k})^*$. Finally,
recall that the projective norm  is the largest cross norm (see Proposition 6.1 in \cite{R:13}). This in particular implies that the space of coefficients obtained considering the projective norm
is larger than the space of coefficients obtained considering any other crossnorm.

Our goal is now to cast \eqref{eq:polyforward} in the polynomial framework. Consider the operator 
\begin{equation}\label{eqn5}
\mathcal{A}: \operatorname{dom}( \mathcal{A})\to \B,\qquad \Acal\y:=\y',
\end{equation}
where $ \operatorname{dom}( \mathcal{A}):=\{\y\in \B\colon \Acal \y\in \B\}$.
In order to define an appropriate set $D$ of coefficients, we have to make sure that the adjoint of $\Acal$,
denoted by $\mathcal{A}^*$, is well defined on $D$. Let therefore
%Set now $B':=\{\y\in\B\ :\ \y'\in\B\}$,
$$\operatorname{dom}(\mathcal{A}^*):=\left\{
\A\in\B \colon \exists C \geq 0 \text{ s.t.~}   \left|\left\langle \A, \mathcal{A} y\right\rangle_\alpha\right|
\leq C \|\y\|_\alpha\text{ for all }\y\in \operatorname{dom}(\mathcal{A})\right\},$$
and  define the adjoint  $\Acal^*:\operatorname{dom}( \mathcal{A^*})\to\B$ 
as usual
as the linear operator that is uniquely determined by
$$\langle \Acal^*\A,\y\rangle_\alpha
=\langle \A,\Acal y\rangle_\alpha,\qquad \A\in\operatorname{dom}( \mathcal{A}^*), \y\in \operatorname{dom}( \mathcal{A}) .$$

Fix then $D\subseteq \operatorname{dom}( \mathcal{A}^*)$ and  let $(\lambda_t)_{t\geq0}$ be a polynomial diffusion corresponding to the linear operator $L:P^D\to P$ given by
\begin{align}\label{eq:polyoperator}
Lp(\y):=\langle \Acal^*(\partial p(\y)),\y\rangle_\alpha+  \frac 1 2 \sum_{i=0}^2\langle Q^i(\partial^2p(\y)),\y^{\otimes i}\rangle_\alpha.
\end{align}
for some linear operators $Q^i:D\otimes D\to(\B^{\otimes i})^*$.

In the next lemma we establish the connection between forward variance curve models as given in \eqref{eq:polyforward}
and such polynomial diffusions.

\begin{lemma}\label{lem3}

Let
$(M_t)_{t\geq0}$ be a $\B$-valued square integrable continuous martingale. 
Let $(\lambda_t)_{t\geq0}$ be an analytically (and also probabilistically) weak solution of the SPDE given by
\begin{align}\label{eq:SPDE}
d\lambda_t=\Acal\lambda_tdt+dM_t,
\end{align}
i.e.~$
\langle \A, \lambda_t \rangle_{\alpha}=\langle \A,\lambda_0\rangle_{\alpha} + \int_0^t \langle \mathcal{A}^* \A , \lambda_s \rangle_{\alpha} dt+\langle \A, M_t \rangle _{\alpha},
$
for each $a\in D \subseteq \operatorname{dom}(\mathcal{A}^*)$.
 Suppose that the dynamics of the quadratic variation process are given by  $$d[\langle \A,\lambda_\cdot\rangle_\alpha,\langle \A,\lambda_\cdot\rangle_\alpha]_t=[\langle \A,M_\cdot \rangle_\alpha,\langle \A,M_\cdot\rangle_\alpha]_t
=\sum_{i=0}^2\langle Q^i(\A\otimes \A),\lambda_t^{\otimes i}\rangle_\alpha dt.
$$ 
Then $(\lambda_t)_{t\geq0}$ is a polynomial process corresponding to $L$ given in \eqref{eq:polyoperator}.
\end{lemma}
\begin{proof}
We have to prove that $(\lambda_t)_{t \geq 0}$ is a solution to the martingale problem for the polynomial operator $L$. The existence of a c\`adl\`ag version of $t \mapsto p(\lambda_t)$ and $t \mapsto Lp(\lambda_t)$ for $p\in P^D$ is clear since $M$ is continuous.
Moreover, by It\^o's formula
$N^{p}$ as of \eqref{IIeqnN} is a martingale for every $p \in P^D$, which proves the assertion.
\end{proof}

\begin{remark}\phantomsection\label{rem:mildconcept}
\begin{enumerate}
\item\label{it1i}
By resorting to a mild solution of \eqref{eq:SPDE}, we can weaken the assumptions on $M$ to allow for instance for 
\begin{align}\label{eq:roughmartingale}
M_t(x)=\int_0^t K(x)\lambda_s(x) dW_s,
\end{align}
where $K \in L^2_{\text{loc}}(\R_+)$ is a fractional kernel $K(t)\approx t^{\alpha}$ for $\alpha \in (-\frac{1}{2}, 0)$ having a singularity at $0$ so that $M$ is not an element in $B$. This form of $M$ is an important example for rough volatility modeling as we will see in Sections \ref{ex6} and \ref{ex5} below.
By denoting the shift semigroup by $(S_t)_{t \geq 0}$, a weakly mild solution of  \eqref{eq:SPDE} is given by
\[
\langle \A, \lambda_t \rangle_{\alpha}=\langle \A,S_t\lambda_0\rangle_{\alpha} +\int_0^t \langle \A, S_{t-s}dM_s \rangle _{\alpha}, \quad \A \in B,
\]
where it is just required that $\int_0^t \langle \A, S_{t-s}dM_s \rangle _{\alpha}$ is well-defined. This is for instance the case if we consider the example given in \eqref{eq:roughmartingale} where we get
$$\int_0^t \langle \A, S_{t-s}dM_s \rangle _{\alpha}=\int_0^t \langle \A, K(t-s+\fdot)\lambda_t(t-s+\fdot) \rangle _{\alpha}dW_s.$$

\item \label{rem:fractional} 
One can however still consider a weak solution concept when restricting the set $D \subseteq\operatorname{dom}(\mathcal{A}^*)$ to elements $\A \in D$ for which we can make sense out of $\langle \A, M_t\rangle_{\alpha}$, even if $M_t$ is not in $B$. In order to deal with kernels $K$ with a singularity at $0$ as in \eqref{eq:roughmartingale}, we shall introduce some new notation.  We let $B_z:=\{\y:S_z\y \in\B\}$ and define $\widetilde \B :=\bigcap_{z>0}\B_z$. For $\A\in \B$ and $K\in\widetilde \B$ set 
$\langle\langle \A,K\rangle\rangle_\alpha:=\lim_{z\to\infty}\langle \A,K(\fdot+z)\rangle_{\alpha},$
whenever the limit exists and $\langle\langle \A,K\rangle\rangle_\alpha:=\infty$ if the limit does not exists. Observe in particular that $\langle\langle \A,\y\rangle\rangle_{\alpha}=\langle \A,\y\rangle_{\alpha}$ for each $\y\in \B$.

\item\label{it1iii}
  Weakly mild solutions are then actually weak solutions when restricting the set $D \subseteq\operatorname{dom}(\mathcal{A}^*)$ to elements $\A \in D$ for which $\langle\langle  \A, M_t \rangle \rangle_{\alpha}$ is well-defined (see for instance Example \ref{ex6} below). More precisely, a weakly mild solution $(\lambda_t)_{t\geq0}$ of \eqref{eq:SPDE} satisfies
 \[
\langle \A, \lambda_t \rangle_{\alpha}=\langle \A,\lambda_0\rangle_{\alpha} + \int_0^t \langle \mathcal{A}^* \A , \lambda_s \rangle_{\alpha} dt+\langle\langle \A, M_t \rangle\rangle _{\alpha}.
\]
 This follows e.g. from the results in~\cite[Proposition 6.3]{Kunze:13}. Hence, a weakly mild solution solves the martingale problem for such a restricted set $D$.
\end{enumerate}
\end{remark}

\subsubsection{Moments of the VIX-Index}\label{sec:vix}

As concrete application of the moment formula in the case of polynomial forward variance curve models we have pricing of VIX options in mind.
Define the VIX at time $t$
via the continuous time monitoring formula
\[
VIX_t^2=
\frac{1}{\Delta}\int_0^{\Delta} \lambda_t(x) dx,
\]
where $\Delta$ is typically 30 days (see e.g.~\cite{HJT:18}). 
This is clearly a linear functional of $\lambda$, explicitly
\begin{align*}
VIX_t^2=\langle \widehat\A,\lambda_t\rangle_\alpha
\end{align*}
for  $\widehat a(x):=1+\frac 1 \Delta \int_0^x(\Delta-\Delta\land z)\alpha(z)^{-1}dz$, which lies in $\operatorname{dom}(\mathcal{A}^*)$ and also in the
 the subsets $D\subseteq \operatorname{dom}(\mathcal{A}^*)$ that we shall consider below. 
 The risk neutral valuation formula for an option on VIX with payoff $\phi$ (for the VIX future $\phi(x) =\sqrt{x}$)
is then 
\begin{equation}\label{eqn201}
\mathbb{E}\left[ \phi(VIX^2_t) \right]= \mathbb{E}\left[ \phi (\frac{1}{\Delta}\int_0^{\Delta} \lambda_t(x) dx) \right]= \E[\phi(\langle \widehat \A,\lambda_t\rangle_\alpha)].
\end{equation}
Modulo technicalities, the polynomial property of $(\lambda_t)_{t\geq0}$ implies that for $\phi(x)=x^k$ the expression in \eqref{eqn201} can be computed by solving a system of infinite dimensional linear ODEs.
Below we shall analyze two concrete specifications of polynomial forward variance models, namely the (rough) Bergomi model (see Section \ref{ex6}) and a polynomial Volterra model (see Section \ref{ex5}), where these technical conditions are satisfied. In both cases we give an explicit formula (up to a Lebesgue integration on $\mathbb{R}^k$) for the $2k$-th moment of the VIX.
The following gives a road map for our proof strategy, following the idea explained in Remark~\ref{rem8}. For $\widehat \A$, we write
\begin{equation}\label{eqn14}
 \langle\widehat \A^{\otimes k},f\rangle_\alpha:=\frac 1 {\Delta^k}\int_{[0,\Delta]^k}f(x)dx
\end{equation}
for each $f:\R_+^k\to\R$ such that \eqref{eqn14} is well-defined. Note that this defines a  pairing on the tensor products only for the specific element $\widehat \A^{\otimes k}$ which then coincides with $\langle\langle \widehat \A,\y\rangle\rangle_\alpha^k$ for $f:=\y^{\otimes k}$ and $\y\in \widetilde \B$, and also with $\langle \widehat \A,\y\rangle_\alpha^k$ for $f:=\y^{\otimes k}$ and $\y\in  \B$.
\begin{enumerate}
\item In both examples, the polynomial operators are homogenous, i.e., it maps homogeneous polynomials of degree $k$ to  homogeneous polynomials of degree $k$. Therefore we can express the $n$-th bidual operator $M_n$ as  $M_n\vec y=(\Mcal_0\y_0,\ldots,\Mcal_n\y_n)$ (see Remark~\ref{rem33} for the $\Mcal$ notation).

\item 
We observe that for each $\y\in (\operatorname{dom}(\mathcal{A}))^{\otimes k}$ the corresponding PDE 
\begin{equation}\label{eqn204}
\partial_tm_t=\Mcal_k m_t,\qquad m_0=\y
\end{equation}
 has a classical strong solution on $E^k\times [0,T]$ for some $E\subseteq \R_+$ such that $E\cap[0,\Delta+T]$ has full Lebesgue mass. We denote it by $(Z_t\y)_{t\geq0}$ for some $Z_t\y:E^k\to \R$ such that \eqref{eqn14} is finite and
\begin{align} \label{eq:linfunc}
 \langle\widehat \A^{\otimes k},Z_t\y\rangle_\alpha\leq c_t \|y\|_\times
 \end{align}
  for some $t$-dependent constant $c_t$ not depending on $\y$. 
  
 In particular, \eqref{eqn204} corresponds to a Cauchy problem associated to an $\mathbb{R}^k$-dimensional Markov process and $Z_ty$ is the unique classical solution thereof.
It is however not clear at that point that we have weak solutions in the sense of Definition~\ref{def:solweak} and that the conditions of Corollary~\ref{cor11} are satisfied for $\Hcal=
\{\widehat{a}^{\otimes k}\}$.  Therefore we cannot directly conclude that $\langle \widehat a^{\otimes k} , Z_ty \rangle_{\alpha}= \mathbb{E}[VIX^{2k}]$.
 \item The next step consist in using $(Z_ty)_{t\ge0}$ to construct an ansatz for the
  solution $(a_t)_{t \geq 0}$ of the dual system of ODEs. 
  Due to \eqref{eq:linfunc} we can define $\A_t\in (\B^{\otimes k})^*$ as $\langle \A_t,\y\rangle_\alpha:=\langle\widehat \A^{\otimes k},Z_t\y\rangle_\alpha$ for each $\y\in \B^{\otimes k}$. 
Here, $\A_t$ denotes a candidate solution (in the sense of Definition \ref{def:sol}) of $\partial_t\A_t=\Lcal_k\A_t$ for $\A_0=\widehat \A^{\otimes k}$. 
\item Finally, we verify that $(\A_t)_{t\geq0}$ satisfies the conditions of the dual moment formula in Theorem \ref{thm1} and we can thus conclude that
$$\E[ (VIX_t^2)^k|\lambda_0=\y]=\E[\langle \widehat a^{\otimes k},\lambda_t^{\otimes k}\rangle_\alpha|\lambda_0=\y]=\langle \A_t,\y^{\otimes k}\rangle_\alpha=\langle\widehat \A^{\otimes k},Z_t\y^{\otimes k}\rangle_\alpha.$$
\end{enumerate}

We shall now specify the two concrete examples and deploy the above program for the computation of the VIX-moments.

\subsubsection{The (rough) Bergomi model and its VIX moments}\label{ex6}

 The first example corresponds to the Bergomi model, either in its rough form (e.g., \cite{BFG:16, HJT:18}) or in the original form depending on the choice of the kernel.

\paragraph{Model specification}
Recall the notation of Remark~\ref{rem:mildconcept}\ref{rem:fractional}. Let $K_{\ell} \in L^2_{\text{loc}}$, $\ell=1, \ldots, m$, denote some (potentially fractional) kernels such that $K_{\ell}\in \widetilde{\B}$ and 
\begin{equation}\label{eqn209}
\sum_{\ell=1}^m\int_0^T\sup_{t\in[0,T]} K_{\ell}(x+t)^2dx<\infty
\end{equation}
for each $T>0$.
Define
\[
D:= \Big\{\A \in \operatorname{dom}(\mathcal{A}^*)\colon \exists C \geq 0 \text{ s.t. } \sum_{\ell=1}^m|\langle\langle \A,K_\ell\y\rangle\rangle_\alpha|\leq C\|\y\|_\alpha\text{ for all } y\in B \Big\}.
\]
Consider \eqref{eq:polyoperator} with
$Q^0=0$, $Q^1=0$, and $Q^2:D\otimes D\to(\B\otimes \B)^*$ being uniquely determined by
$\langle Q^2 (\A\otimes \A),\y\otimes \y\rangle_\alpha=\sum_{\ell=1}^m\langle\langle \A,K_\ell\y\rangle\rangle_\alpha^2,$  for $\A\in D$ and $\y\in B.$
The corresponding SPDE \eqref{eq:SPDE} can then be realized as follows
\begin{align}\label{eq:BergomiSPDE}
d\lambda_t(x) =  \Acal\lambda_t(x)dt+\sum_{\ell=1}^m K_\ell(x) \lambda_t (x) dB_t^\ell,
\end{align}
where $B^1,\ldots,B^m$ are $m$ independent Brownian motions. Even though we here do not speak about existence of solutions to this equation,  the solution concept that we have in mind is a mild one as outlined in Remark \ref{rem:mildconcept}\ref{it1i}. 
Note in particular that for $a \in B$ 
\[
\int_0^t \langle a , S_{t-s} dM_s\rangle_{\alpha}=  \sum_{\ell=1}^m \int_0^t\langle a , S_{t-s}  K_\ell \lambda_s  dB_s^\ell\rangle_{\alpha}=
 \sum_{\ell=1}^m \int_0^t\langle a , K_\ell(t-s+\cdot) \lambda_s(t-s+\cdot)  \rangle_{\alpha}dB_s^\ell
\]
is well defined since $K_\ell(t-s+\cdot) \lambda_t(t-s+\cdot)$ takes values in $B$ for all $0 \leq  s < t$. To see that this setting includes the (rough) Bergomi model, compare for instance with \cite{HJT:18}. 
The corresponding $k$-th dual operator satisfies the condition of Remark~\ref{rem3} and can thus be written as
$L_k\vec \A=(\Lcal_0\A_0,\ldots, \Lcal_k\A_k)$, where
$$
\langle\Lcal_j\A^{\otimes j},\y^{\otimes j}\rangle_\alpha:=j\langle\Acal^*\A,\y\rangle_\alpha\langle\A,y\rangle_\alpha^{j-1}
+\frac {j(j-1)}2
 \sum_{\ell=1}^m\langle\langle \A,K_\ell\y\rangle\rangle_\alpha^2\langle \A,\y\rangle_\alpha^{j-2}.
 $$
 
\paragraph{VIX-moments}

We now consider the computation of the moments of the VIX 
in these models. An explicit formula is given the following proposition, whose proof is postponed to Section~\ref{secB}. Recall the notion of a weakly mild solution from Remark~\ref{rem:mildconcept}\ref{it1i}.

\begin{proposition}\label{lem7}
Let $(\lambda_t)_{t \geq 0}$ be a weakly mild solution of \eqref{eq:BergomiSPDE} and let
$k \in \mathbb{N}$. Assume that $\E[\sup_{t\leq T}\|\lambda_t\|_\alpha^{2k}]<\infty$ and set $
V_k(x)=\sum_{\ell=1}^m \sum_{i<j}K_\ell(x_i)K_\ell(x_j)
$ for $x=(x_1, \ldots, x_k)$. 
Then for each $\lambda_0\in \B$ we have
$$\mathbb{E}\left[ (VIX^2_t)^k|\lambda_0 \right]
=\frac 1 {\Delta^k}\int_{[0,\Delta]^k}\prod_{i=1}^k\lambda_0(x_i+t)e^{\int_0^tV_k( x+\tau1)d\tau}d x,
$$
where   $1$ denotes the vector consisting of ones.
\end{proposition}

In the following two examples we give concrete specifications of the kernels.

\begin{example}
Applying  Proposition~\ref{lem7} now to  classical Bergomi model (see \cite{B:04, B:05}), where we have $K_\ell(x)=\omega_\ell e^{-\gamma_\ell x}$ for some constant $\omega_{\ell}$ and $\gamma_{\ell}$  yields
$$\mathbb{E}\left[ (VIX^2_t)^k|\lambda_0 \right]=\frac 1 {\Delta^k}\int_{[0,\Delta]^k}\prod_{i=1}^k\lambda_0(x_i+t)\prod_{i<j}e^{\sum_{\ell=1}^m \frac{\omega_\ell^2}{2\gamma_\ell}(1-e^{-2\gamma_\ell t})e^{-\gamma_\ell(x_i+x_j)}}d\vec x.
$$
\end{example}

\begin{example}
In the case of the rough Bergomi model (see \cite{BFG:16}) we have that $m=1$ and $K_1(x)=x^{H-1/2}$ for $H\in(0,1/2)$ (modulo a multiplicative constant). In this case we get
$$\mathbb{E}\left[ (VIX^2_t)^k|\lambda_0 \right]=\frac 1 {\Delta^k}\int_{[0,\Delta]^k}\prod_{i=1}^k\lambda_0(x_i+t)\prod_{i<j}e^{\int_0^t\big((x_i+\tau)(x_j+\tau)\big)^{H-1/2}d\tau}d x.$$
An inspection of this expression yields
$
\mathbb{E}\left[ (VIX^2_t)^k|\lambda_0 \right]
\leq  (VIX^2_{0,t})^k e^{\frac{k(k-1)}2\int_0^{t}\tau^{2(H-1/2)}d\tau}\\
$
for $VIX^2_{0,t}:=\frac 1 \Delta\int_0^\Delta\lambda_0(x+t)dx$,
proving 
 that the moments of $VIX^2_t$ are bounded from above by the moments of  a log-normal random variable $\overline X^2$ with parameters  $\mu=\ln(VIX^2_{0,t})-t^{2H}/4H$ and $\sigma^2=t^{2H}/2H$. Similarly since
$
\mathbb{E}\left[ (VIX^2_t)^k|\lambda_0 \right]
\geq  (VIX^2_{0,t})^k e^{\frac{k(k-1)}2\int_0^{t}(\Delta+\tau)^{2(H-1/2)}d\tau}\\
$
we also get that the moments of $VIX^2_t$ are bounded
  from below
 by the moments of  a log-normal random variable $\underline X^2$ with parameters  $\mu=\ln(VIX^2_{0,t})-((t+\Delta)^{2H}-\Delta^{2H})/4H$ and $\sigma^2=((t+\Delta)^{2H}-\Delta^{2H})/2H$.
This type of relation to log normal random variables has been used  in \cite{HJT:18}, where log-normal control variates are employed for variance reduction in Monte Carlo simulations.

Finally, observe that from the proof of Proposition~\ref{lem7} we know that $m_t(x):=\lambda_0(x+t1)e^{\int_0^tV_k(x+\tau1)d\tau}$ is  a candidate solution of the bidual system of ODEs corresponding to  $\Mcal_k$. This implies that $\langle \A^{\otimes k},m_t\rangle_\alpha$, whenever it is well defined, heuristically coincides with $\E[\langle \A,\lambda_t\rangle_\alpha^k|\lambda_0]$. For the moments of the spot volatility these heuristics therefore lead to
 $\E[\lambda_t(0)^k|\lambda_0]=\lambda_0(t)^ke^{\frac{k(k-1)}{4H}t^{2H}}$.
\end{example}

\subsubsection{A polynomial Volterra model and its VIX moments}\label{ex5}

The following example corresponds to a polynomial Volterra process for the spot variance and extends therefore (rough) affine models and the Jacobi stochastic volatility model of \cite{AF:16}. As explained below we understand here stochastic Volterra processes in the sense of \cite{ACLP:19}. 

\paragraph{Model specification}
Let $K \in L^2_{\text{loc}}$ denote again some (potentially fractional) kernel such that $K \in \widetilde \B$ in sense of Remark~\ref{rem:mildconcept}\ref{rem:fractional}. Define $D$ as
\[
D:= \left\{\A \in \operatorname{dom}(\mathcal{A}^*)\, | \, \langle \langle \A,K \rangle\rangle_\alpha|< \infty \right\}
\]
and consider \eqref{eq:polyoperator}
with $Q^i: D \otimes D \to (\B^{\otimes i})^*$ given by
$Q^i(\A \otimes \A)=c_i \langle\langle \A, K \rangle\rangle_{\alpha}^21^{\otimes i}$
 for some constants  $c_2$, $c_1$, $c_0$. Since $\langle 1, y \rangle_{\alpha}=y(0)$, the corresponding SPDE \eqref{eq:SPDE} can then be realized by
\begin{align}\label{eq:VolterraSPDE}
d\lambda_t(x) =  \Acal\lambda_t(x)dt+ K(x) \sqrt{C(\lambda_t(0))} dB_t,
\end{align}
where $B$ is a Brownian motion and  $C(v)= c_2 v^2+ c_1 v + c_0$. Again even though we do not treat existence of solutions, the solution concept that we have in mind is a mild one as outlined in Remark \ref{rem:mildconcept}\ref{it1i}.
Note in particular that for $a \in B$ 
\begin{align*}
\langle a, \lambda_t \rangle_{\alpha}&= \langle a, S_t\lambda_0 \rangle_{\alpha}+
\int_0^t \langle a, S_{t-s}K \rangle_{\alpha} C(\lambda_s(0)) dB_s\\&= \langle a, \lambda_0(t+\cdot) \rangle_{\alpha}+
\int_0^t \langle a, K(t-s +\cdot) \rangle_{\alpha} C(\lambda_s(0)) dB_s.
\end{align*}
Similarly as in \cite{CT:18}, choosing $a=1$ and setting $V_t=\langle 1, \lambda_t \rangle_{\alpha}=\lambda_t(0)$, then yields the announced Volterra equation
$
V_t=\lambda_0(t)+ \int_0^t K(t-s) C(V_s) dB_s.
$
For $c_0=c_1=c_2-1=0$, this corresponds to a \emph{Volterra geometric Brownian motion}, i.e.
$
V_t=\lambda_0(t)+ \int_0^t K(t-s) V_s dB_s.
$
In this particular parameter case the corresponding $k$-th dual operator satisfies the conditions of Remark~\ref{rem3} and can thus be written as
$L_k\vec \A=(\Lcal_0\A_0,\ldots, \Lcal_k\A_k)$, where
$$\Lcal_j\A^{\otimes j}=j(\Acal^*\A)\otimes \A^{\otimes (j-1)}
+\frac {j(j-1)}2
 \langle\langle \A, K \rangle\rangle_{\alpha}^2 1\otimes 1\otimes \A^{\otimes (j-2)}.
$$

\paragraph{VIX-moments}
We now compute the VIX moments in the case of the Volterra geometric Brownian motion, i.e. $c_0=c_1=c_2-1=0$
and we additionally suppose that $K\in \B$. The proof of the following proposition can be found in Section~\ref{secB}. Recall the notion of a weak solution from Lemma~\ref{lem3}, which we can apply here since we suppose that $K \in B$. Note also that $K$ is automatically bounded due to the choice of the weight function $\alpha$.

\begin{proposition}\label{prop1}
Let $(\lambda_t)_{t \geq 0}$ be a weak solution of \eqref{eq:VolterraSPDE} with $c_0=c_1=c_2-1=0$ and $K \in B$. Let $k \in \mathbb{N}$ and
assume that $\E[\sup_{t\leq T}\|\lambda_t\|_\alpha^{2k}]<\infty$. 
Then
$$\mathbb{E}\left[ (VIX^2_t)^k|\lambda_0 \right]
=\E[e^{\int_0^tV_k(X^{(k)}_\tau)d\tau}\lambda_0^{\otimes k}(X^{(k)}_t)|X^{(k)}_0\sim \Ucal([0,\Delta]^k)],
$$
where $V_k(x)=\sum_{i<j}K(x_i)K(x_j)$ and  $(X^{(k)}_t)_{t\geq0}$ is the $\R^k$-valued process generated by
 $$\Gcal^{(k)} f(x)=1^\top\nabla  f( x)
+\int f( x+\xi)-f( x)\nu(x,d\xi), \qquad  x\in \R^k$$
for $\nu( x ,\fdot)=\sum_{i<j}K(x_i)K(x_j)\delta_{(..,0,-x_i,0..0,-x_j,0..)}$.
\end{proposition}

In the following we analyze some specific parametrization and specific moments.

\begin{example}
Let  $K(x)=\omega e^{-\gamma x}$ and $\lambda_0=c(1-e^{-\gamma x})+e^{-\gamma x}V_0$. Then 
the assumption of the above proposition, namely $\E[\sup_{t\leq T}\|\lambda_t\|_\alpha^{2k}]<\infty$, is  automatically satisfied. Indeed in this case 
$$\lambda_t(x)=\E[V_{t+x}|\Fcal_t]=c(1-e^{-\gamma x})+e^{-\gamma x}V_t,$$
where $V_t$ solves $dV_t=\gamma(c- V_t)dt+\omega V_tdW_t$. For a general initial condition $\lambda_0$, the equation for $V_t$ becomes
$$dV_t=\gamma(\lambda_0(t)+\frac 1 \gamma\lambda_0'(t)- V_t)dt+\omega V_tdW_t.$$
Finally, observe that assuming that  $\lambda_0=be^{-\gamma x}$ for some constant $b\in \R$ we get
$$\mathbb{E}\left[ VIX^{2k}_t |\lambda_0=be^{-\gamma x}\right]=(b(1-e^{-\gamma\Delta}))^k(\gamma\Delta)^{-k}e^{-(k\gamma-\frac{k(k-1)}{2}\omega^2)t}.$$

\end{example}

\begin{remark}
Proposition~\ref{prop1} only treats the case when $K \in B$. We however expect that the result still holds true if 
$K\in\widetilde \B\setminus \B$ and $K \in L^2_{\text{loc}}$.  Analyzing the generator $\Gcal^{(k)}$ we can deduce that the process $(X^{(k)}_t)_{t\geq0}$ has constant drift 1 till the first jump. This occurs at an exponential time. When a jump occurs two of the components of the process jump to 0. 

\end{remark}

\subsection{Signature of Brownian motion}\label{sec:sig}

In the following section we show that the signature process of a $d$-dimensional Brownian motion, that is the sequence of iterated integrals (in the Stratonovich sense), is a polynomial process. We then exploit the polynomial machinery to compute its expected signature, a well-known formula which can for instance be found in \cite{FH:14}.
Before introducing the mathematical framework let us here briefly outline the relevance of the signature of general stochastic processes.
The signature of a path, first studied by \cite{C:57, C:77}, is a highly important object in rough path theory (\cite{L:98}).  This is explained by the following three facts:

\begin{itemize}
\item The signature of a path of bounded variation uniquely determines the path up to the tree-like equivalence (see, e.g.~Theorem 2.13 in \cite{LLN:13}).

\item 
  Under certain regularity conditions, the expected signature of a stochastic process determines 
the law of the signature  (see Theorem 6.1 in \cite{CL:16}).

\item Every continuous function of the signature can  be approximated by a linear function of the signature arbitrarily well (on compacts).
\end{itemize}

These three properties have led the authors of \cite{LLN:13} to introduce the so-called expected signature model which is nothing else than a linear regression model for the signature of a stochastic process $Y$ on the signature of a stochastic process $X$.

In practice the stochastic process $Y$ is often a stochastic differential equation driven by $X$. The (generically non-linear and non Lipschitz) functional relationship of $X \mapsto Y$, can then be explained by a linear map of the signature of $X$. %\footnote{Note that $Y$ itself is the first component of the signature.}  
Considering for $X$ a $d$-dimensional Brownian motion $B$, this means that any $n$ dimensional SDE of the form
\begin{align}\label{eq:Y}
dY_t= \sum_{i=1}^d V_i(Y_t) \circ B_t^i, \quad Y_0=y,
\end{align}
with analytic vector fields $V_i: \mathbb{R}^n \to \mathbb{R}^n$ and where $\circ$ denotes the Stratonovich integral 
can be represented by a linear map of the signature of $B$.
Since the signature process of $B$ is a polynomial process as outlined below, this then also translates to the signature process of $Y$. In this sense we encounter a surprising universality of the polynomial class. In a subsequent paper we shall
derive the polynomial property of the signature of processes of form \eqref{eq:Y} directly and provide a procedure how to compute the expected signature. This has applications  for instance for the generalized method of moments on a process level as considered for instance in \cite{PL:11}. Here, we consider the important case of Brownian motion.

Set\footnote{Here we use $\ootimes$ instead of $\otimes$ to emphasize that it denotes a nonsymmetric tensor product.}  $T(\R^d):=\bigoplus_{n=0}^\infty(\R^d)^{\ootimes n}$ and  $T^N(\R^d):=\bigoplus_{n=0}^N(\R^d)^{\ootimes n}$, and recall that $(\R^d)^{\ootimes n}\cong\R^{d^n}$.  
In particular,  $(\R^d)^{\ootimes 0}\cong\R $.
For each $ \y\in T^N(\R^d)$  let then $\y_{n,{\bf i}}$ denote the components of $ \y=(\y_0,\ldots,\y_N)$ meaning that 
$$ \y=\y_0+\sum_{n=1}^N\sum_{{\bf i}\in \Ical_n} \y_{n,{\bf i}} e_{i_1}\ootimes\cdots\ootimes e_{i_n},\qquad \Ical_n:=\{1,\ldots,d\}^n, $$
where $e_i$ denotes the $i$-th unit vector in $\R^d$.
We are now interpreting the signature of a $d$-dimensional Brownian motion as a polynomial process taking values on the Banach space $B=T^N(\R^d)$. The  dual structure corresponding to $B$ is then given by $B^*=T^N(\R^d)$ with pairing $\langle a, y \rangle = a_0y_0+\sum_{n=1}^N\sum_{{\bf i}\in \Ical_n} y_{n,{\bf i}} a_{n,{\bf i}},$ for  $\Ical_n=\{1,\ldots,d\}^n$.

%One can then see that the notion of polynomials on $T^N(\R)$ coincides with the notion of polynomial on $\R^{\widetilde N}$  for $\widetilde N:=\sum_{n=0}^Nd^n$, meaning that
%$$P=\{p:T^N(\R)\to \R\colon p( \y)=q(q_0( \y_0),\ldots,q_N(\y_N))\text{ for some polynomials $q$ on }\R^{N+1}\text{ and $q_n$ on }\R^{d^n}\}.$$

Let now $(B_t)_{t\in[0,T]}$ be a $d$-dimensional Brownian motion.
Recall that denoting by $\circ$ the Stratonovich integral we can write
$$\int H_s\ootimes \circ dB_s=\sum_{i=1}^{\ell}\sum_{j=1}^d\Big(\int H_s^i \circ dB_s^j\Big) e_i\ootimes e_j,$$
for each $\R^\ell$-valued process $H$ such that the right hand side is well defined. Define then
$$S_T^{(0)}=1,\qquad S_T^{(n)}:=\int_{0<t_1<\cdots<t_n<T}\circ dB_{t_1}\ootimes\cdots\ootimes\circ dB_{t_n}\quad n\geq1.$$
The signature $S(B)_{0,T}\in T(\R)$  and the truncated signature $S(B)^N_{0,T}\in T^N(\R^d)$ of $B$ are then given by
$$S(B)_{0,T}:=(S_T^{(0)},S_T^{(1)},S_T^{(2)},\ldots)\quad\text{ and }\quad S(B)^N_{0,T}:=(S_T^{(0)},S_T^{(1)},\ldots, S_T^{(N)}),$$
respectively. Observe that by the definition of the Stratonovich integral we have that
$$dS_t^{(n)}=\frac 1 2 \sum_{i=1}^d\big(S_t^{{(n-2)}}\ootimes  e_i\ootimes e_i\big)dt+S_t^{{(n-1)}}\ootimes dB_t.$$
An  application of the It\^o formula yields that the generator $L:P\to P$  of $(S(B)^N_{0,t})_{t\geq0}$ is given by
\begin{align*}
Lp( \y)&=\frac 1 2 \sum_{n=2}^N\sum_{{\bf i}\in \Ical_n}
\y_{n-2,i_1\cdots i_{n-2}}1_{\{i_{n-1}=i_n\}}
\frac d{d\y_{n,{\bf i}}}p( \y)\\
&\qquad+\frac 1 2 \sum_{n,m=1}^N\sum_{{\bf i} \in \Ical_n}\sum_{{\bf j} \in \Ical_m}
\y_{n-1,i_1\cdots i_{n-1}}
\y_{m-1,j_1\cdots j_{m-1}}1_{\{i_{n}=j_m\}}
\frac {d^2}{d\y_{n,{\bf i}}d\y_{m,{\bf j}}}p( \y)
\end{align*}
for each $p\in P$, showing that $(S(B)^N_{0,t})_{t\geq0}$ is a polynomial diffusion. Compute then 
$$
L_1(e_{i_1}\ootimes \cdots\ootimes e_{i_n})=\begin{cases}
\frac 12 1_{\{i_{n-1}=i_n\}}e_{i_1}\ootimes \cdots\ootimes e_{i_{n-2}}&\text{for $n\geq2$},\\
%\frac 12 1_{\{i_{n-1}=i_n\}}&\text{for $n=2$}\\
0&\text{else},
\end{cases}
$$
and observe that for $n$ even we have that
\begin{align} 
\exp(tL_1)(e_{i_1}\ootimes \cdots\ootimes e_{i_n})
&=\sum_{\ell=0}^\infty \frac{t^\ell}{\ell!}L_1^\ell(e_{i_1}\ootimes \cdots\ootimes e_{i_n})\notag\\
&=\sum_{\ell=0}^{n/2} \frac{(t/2)^\ell}{\ell!}\prod_{k=0}^{\ell-1}1_{\{i_{n-2k}=i_{n-2k-1}\}}e_{i_1}\ootimes \cdots\ootimes e_{i_{n-2\ell}}.\label{eqn500}
\end{align}
Note here that $L^{\ell}_1$ means an $\ell$-fold application of $L_1$ and the empty product is equal to $1$.
Let us now compute 
the expectation of the $(i_1\cdots i_n)$-component of $S_t^{(n)}$, i.e.,
$\mathbb{E}[ \langle e_{i_1} \ootimes \cdots\ootimes e_{i_n}, S(B)^{N}_{0,t} \rangle ]$.
By the dual moment formula this is equal to $ \langle a_t, S(B)^{N}_{0,0} \rangle $, where $a_t$ is the solution of 
$
\partial_t a_t =L_1 a_t,$ for $a_0 =e_{i_1} \ootimes \cdots\ootimes e_{i_n},
$
which is given by the exponential in \eqref{eqn500}.
Since $S(B)^N_{0,0}=1+\sum_{n=1}^N\sum_{{\bf i}\in \Ical_n} 0 e_{i_1}\ootimes\cdots\ootimes e_{i_n}$,
we conclude that
\begin{align}
\mathbb{E}[ \langle e_{i_1} \ootimes \cdots\ootimes e_{i_n}, S(B)^{N}_{0,t} \rangle ]&= \left\langle \exp(tL_1)(e_{i_1}\ootimes \cdots\ootimes e_{i_n}), S(B)^N_{0,0} \right \rangle \notag \\
&= \frac{(t/2)^{(n/2)}}{(n/2)!}\prod_{k=0}^{n/2-1}1_{\{i_{n-2k}=i_{n-2k-1}\}}, \label{eq:indicators}
\end{align}
i.e. the expectation of the $(i_1\cdots i_n)$-component of $S_t^{(n)}$ is the coefficient of the basis element 1 in \eqref{eqn500}.
 A similar reasoning shows that the expectation of the $(i_1\cdots i_n)$-component of $S_t^{(n)}$ is 0 for each $n$ odd.
Let now $n=2k$. Then the only indices  $(i_1\cdots i_{2k})$ which lead to nonzero in \eqref{eq:indicators} are of the form
$(j_1, j_1, j_2, j_2, ...,j_{k}, j_{k})$. Hence the only basis elements of $\E[S_t^{(2k)}]$ whose coefficient is non zero are of the form $ e_{j_1} \ootimes e_{j_1} \ootimes e_{j_2}  \ootimes e_ {j_2}, \cdots,e_{j_{k}}\ootimes e_{j_{k}}$. (Formally) summing over all those yields 
\[
\sum_{j_1=1}^d\sum_{j_2=1}^d \cdots\sum_{j_k=1}^d e_{j_1} \ootimes e_{j_1} \ootimes e_{j_2}  \ootimes e_ {j_2} \cdots e_{j_{k}}\ootimes e_{j_{k}}=(\sum_{i=1}^d e_i\ootimes e_i)^{\ootimes k},
\] 
and we can thus conclude that
$$\E[S_t^{(n)}]=\begin{cases}
\frac{(t/2)^k}{k!}(\sum_{i=1}^d e_i\ootimes e_i)^{\ootimes k} &\text{if $n=2k$ for some $k\in \N\cup{\{0\}}$,}\\
0& \text{otherwise.}
\end{cases}$$
Summing over $n$ yields
$\E[S(B)^N_{0,t}]=\sum_{k=0}^{\lfloor N/2\rfloor}\frac{(t/2)^k}{k!}(\sum_{i=1}^d e_i\ootimes e_i)^{\ootimes k},$
which in particular implies that
$$\E[S(B)_{0,t}]=\sum_{k=0}^{\infty}\frac{(t/2)^k}{k!}(\sum_{i=1}^d e_i\ootimes e_i)^{\ootimes k}=\exp\Big(\frac t 2\sum_{i=1}^d e_i\ootimes e_i\Big).$$
This expression coincides with the result of Theorem~3.9 in \cite{FH:14}.

\appendix

%%%%%SECTION A
\section{Auxiliary results}\label{secA}

We here collect auxiliary results which are needed in Section \ref{IIsec:2}.

\begin{lemma}\label{lem4}
Fix $p(\y)=\sum_{j=0}^k\langle \A_j,\y^{\otimes j}\rangle$ for $\A_j\in(\B^{\otimes j})^*$.  Then $p(\y)=0$ for all $\y\in \B $ if and only if $\A_j=0$ for all $j$.
\end{lemma}
\begin{proof}
Observe that $\langle \A_j,\y^{\otimes j}\rangle=0$ for all $\y\in \B$ if and only if $\langle \A_j,\tilde \y\rangle=0$ for all $\tilde\y\in \B^{\otimes j}$ and thus if and only if $\A_j=0$. Suppose that $p(\y)=0$ for all $\y\in \B $. Then $\A_0=p(0)=0$. Proceeding inductively we can prove that $\langle \A_j,\y^{\otimes j}\rangle=\lim_{\e \to 0}\e^{-j}p(\e \y)=0$ and the first implication follows. The second implication is clear.
\end{proof}

\begin{lemma}\label{lem12}
Let $L\colon P^D\to P$ be an $\Scal$-polynomial operator. Then there exists a $k$-th dual operator for each $k\in\N$.
\end{lemma}
\begin{proof}
We claim that there exists a $\B $-polynomial operator $\widetilde L:P^D\to P$ such that $\widetilde Lp|_\Scal=Lp|_\Scal$ for each $p\in P^D$. In other words, we claim that for each $p\in P^D$ there exists a $q_p\in P$ such that $\deg(q_p) \le \deg(p)$ and $p\mapsto q_p$ is linear.
If this is the case the claim follows by Lemma~\ref{lem4}.

We proceed by induction. Set $P^D_k:=\{p\in P^D:\deg(p)\leq k\}$ and for all $p\in P^D_0$ set $\widetilde Lp\equiv q$ where $q\in \R$ satisfies $Lp|_\Scal\equiv q$. 
Clearly $\widetilde L|_{P^D_0}$ is linear, satisfies the $\B $-polynomial property, and $\widetilde Lp|_\Scal=Lp|_\Scal$ for all $p\in P^D_0$. 
Fix $k\in \N$ and suppose that $\widetilde L|_{P^D_k}$ is linear, satisfies the $\B $-polynomial property, and $\widetilde Lp|_\Scal=Lp|_\Scal$ for all $p\in P^D_k$. 
Consider the set of all pairs $(V,\Lcal)$ of a vector space $V$ such that $P^D_{k}\subseteq V \subseteq P^D_{k+1}$ and $\Lcal: V\to P$ 
is a linear extension of $\widetilde L|_{P^D_k}$  satisfying the $\B $-polynomial property and such that $\Lcal p|_\Scal= Lp|_\Scal$ for all $p\in V$. 
By a standard application of Zorn's lemma, we get that this set has a maximal element $(V,\Lcal)$ with respect to the order relation given by
$$(V_1,\Lcal_1)``\leq"(V_2,\Lcal_2)\qquad :\Leftrightarrow \qquad V_1\subseteq V_2\quad \text{and}\quad\Lcal_2|_{V_1}=\Lcal_1.$$
Assume by contradiction that $V\neq P^D_{k+1}$ and pick $p\in P^D_{k+1}\setminus V$. Since $L$ is $\Scal$-polynomial, there exists a $q\in P$ such that $Lp|_\Scal=q$ and $\deg(q)\leq k+1$. Set $\Lcal p:=q$ and extend $\Lcal$ to $V+\R p$ linearly. Linearity of $L$ and $\Lcal$ yield that $\Lcal\tilde p|_\Scal=L\tilde p|_\Scal$ for all $\tilde p\in V+\R p$. Finally, noting that $\deg(\tilde p+\alpha p)=k+1$ for all $\tilde p\in V$ and $\alpha\neq0$ we can conclude that
$$\deg(\Lcal (\tilde p+\alpha p))=\deg(\Lcal \tilde p +\alpha q)\leq \max(\deg(\Lcal \tilde p), \deg(\alpha q))\leq k+1$$
contradicts the maximality of $(V,\Lcal)$. Setting $\widetilde L p:= \Lcal p$ for all $p\in P^D_{k+1}$ concludes the proof.
\end{proof}

%%%%%SECTION B
\section{Proof of Proposition~\ref{lem7} and Proposition~\ref{prop1}}\label{secB}

We here collect the proofs of the VIX moment formulas stated in Section \ref{sec43}.

\begin{proof}[Proof of Proposition~\ref{lem7}]
Recall that by Remark~\ref{rem:mildconcept}\ref{it1iii}, the process $(\lambda_t)_{t \geq 0}$ solves the martingale problem  for the polynomial operator $L:P^D\to P$.
We follow now the scheme outlined in Section \ref{sec:vix}. Throughout the proof we shall use the following inequalities. Set $C_t:=e^{\frac{k(k-1)}{2}\sum_{\ell=1}^m\int_0^{t+\Delta}K_\ell(\tau)^2d\tau}$ and note that $e^{\int_0^tV_k(x+\tau1)d\tau}\leq C_t<\infty$ for each $t>0$. Observe also that by Lemma~3.2 in \cite{BK:14} we have that
\begin{equation}\label{eqn207}
|\y(x)|\leq \cc\| \y\|_\alpha\qquad\text{and}\qquad \frac 1 \Delta \int_0^\Delta\y'(x+t)dx\leq\|\y\|_\alpha
\end{equation}
for each $\y\in \B$, for some constant $\cc>0$. 
\begin{enumerate}
\item We observe that $M_n\vec y=(\Mcal_0\y_0,\ldots,\Mcal_n\y_n)$
 with
\begin{align*}
\Mcal_k\y( x)&=1^\top\nabla\y(x)+V_k(x)\y(x),\qquad \y\in (\operatorname{dom}(\mathcal{A}))^{\otimes k},
\end{align*}
where $\mathcal{A}$ is given by \eqref{eqn5}. 
\item
 Observe that for each $\y\in (\operatorname{dom}(\mathcal{A}))^{\otimes k}$ the corresponding PDE \eqref{eqn204}
 corresponds to the Cauchy problem of a pure drift process $dX_t= 1 dt$ on $\mathbb{R}^k$ with potential $V_k$. The Feynman Kac formula thus yields that its classical strong solution on $(\R_+\setminus\{0\})^k\times [0,T]$ is given by $Z_ty( x):=\y(x+t1)e^{\int_0^tV_k(x+\tau1)d\tau}$. 
 Moreover, note that by \eqref{eqn14}, definition of $C_t$, and \eqref{eqn207} we get that
$$|\langle \widehat \A^{\otimes k},Z_t\y^{\otimes k}\rangle_\alpha|
=\bigg|\frac 1 {\Delta^k}\int_{[0,\Delta]^k}Z_t\y^{\otimes k}(\vec x)d\vec x\bigg|
\leq \cc ^kC_t \|\y\|_\alpha^k$$
for each $\y\in \B$, proving that $\langle\widehat \A^{\otimes k},Z_t\y\rangle_\alpha\leq \cc ^kC_t \|y\|_\times$ for each $\y\in \B^{\otimes k}$, which is \eqref{eq:linfunc}.
 \item\label{it2iii} Due to this estimate 
  we can now define a candidate solution  $\A_t\in (\B^{\otimes k})^*$ (in the sense of Definition \ref{def:sol}) of $\partial_t\A_t=\Lcal_k\A_t$ for $\A_0=\widehat \A^{\otimes k}$ as
 $$\langle \A_t,\y\rangle_\alpha:=\langle\widehat \A^{\otimes k},Z_t\y\rangle_\alpha
 =\frac 1 {\Delta^k}\int_{[0,\Delta]^k}\y( x+t1)e^{\int_0^tV_k( x+\tau1)d\tau}d x,\qquad \y\in \B^{\otimes k}.$$
\item In order to conclude the proof, we need to check that 
$(\A_t)_{t\geq0}$ satisfies the conditions of the dual moment formula given in Theorem~\ref{thm1}\ref{iti}-\ref{itii}.
Observe that Condition~\ref{iti} of Theorem~\ref{thm1} is satisfied if and only if there is a map $(\A_t)_{t\geq0}\subseteq \Dcal(\Lcal_k)$ satisfying
\begin{align}\label{eq:dualPDEBergomi}
\langle \A_t, y^{\otimes k }\rangle_{\alpha}=\langle \widehat \A, y^{\otimes k }\rangle_{\alpha} +\int_0^t \langle \mathcal{L}_k a_s, y^{\otimes k }\rangle_{\alpha}ds
\end{align}
for all $ y \in B$.
We prove now the following claims:

\emph{Claim: $\A_t\in \Dcal(\Lcal_k)$} \\
To this end we fix $t\geq0$ and we construct a sequence $\A_n=\sum_{i=1}^n\alpha_{n,i}\A_{n,i}^{\otimes k}$ such that $\A_{n,i}\in D$, $\|\A_n-\A_t\|_{*k}\to0$, and $\|\Lcal_k \A_n-\Lcal_k \A_t\|_{*k}\to0$ where $\Lcal_k \A_t\in (\B^{\otimes k})^*$ is uniquely determined by
\begin{equation}\label{eqn9}
\langle\Lcal_k \A_t,\y^{\otimes k}\rangle_\alpha
=\frac 1 {\Delta^k}\int_{[0,\Delta]^k}\Mcal_k\y^{\otimes k}(x+t1)F(x)dx
,\qquad \y\in \operatorname{dom}(\mathcal{A})
\end{equation}
for $F( x):=e^{\int_0^tV_k( x+\tau1)d\tau}$. Observe that the assumptions on $K$ guarantee that  $F$ is symmetric and continuous on $[0,\Delta]^k$. By Stone-Weierstrass theorem we can thus find a sequence $F_n=\sum_{i=1}^n\alpha_{n,i}F_{n,i}^{\otimes k}$ such that $F_{n,i}\in C([0,\Delta])$ and $\e_n:=\|F_n-F\|_{\infty,k}\to0$ for $n\to\infty$, where $\|\fdot\|_{\infty,k}$ denotes the supremum norm on $[0,\Delta]^k$. Set $\A_{n,i}\in \B$ as $\langle\A_{n,i},\y\rangle_\alpha:=\frac 1 \Delta\int_0^\Delta\y(x+t)F_{n,i}(x)dx$. Using \eqref{eqn207} we obtain the following estimates
\begin{align*}
|\langle\A_{n,i},\Acal\y\rangle_\alpha|&\leq \frac 1 \Delta\int_0^\Delta |\Acal\y(x+t)|dx\|F_{n,i}\|_{\infty,1}\leq\|\y\|_\alpha\|F_{n,i}\|_{\infty,1}\\
|\langle\langle\A_{n,i},K_\ell\y\rangle\rangle_\alpha|
&=\lim_{z\to 0}|\langle \A_{n,i},K_\ell(\fdot+z)\y(\fdot+z)\rangle_\alpha|\\
&\leq \frac 1 \Delta\int_0^\Delta |K_\ell(x+t)\y(x+t)||F_{n,i}(x)|dx\\
&\leq \cc\|\y\|_\alpha \|F_{n,i}\|_{\infty,1} \frac 1 \Delta\int_t^{\Delta+t} K_\ell(x)^2dx,
\end{align*}
from which we can conclude that $\A_{n,i}\in D$. Next, using  again \eqref{eqn207} we have for each $\y\in \B$ 
$$|\langle \A_n-\A_t,\y^{\otimes k}\rangle_\alpha|\leq \bigg(\frac 1 \Delta\int_0^\Delta |y(x+t)|dx\bigg)^k\e_n
\leq \cc^k\|y\|_\alpha^k\e_n,$$
which by \eqref{eqn8} proves that $\| \A_n-\A_t\|_{*k}\leq \cc^k\e_n$ and thus that
$\|\A_n-\A_t\|_{*k}\to0$ for $n$ going to infinity. Since for each $\y\in \operatorname{dom}(\mathcal{A})$ we have again by  \eqref{eqn207}
\begin{align*}
|\langle \Lcal_k \A_n&-\Lcal_k \A_t,\y^{\otimes k}\rangle_\alpha|
\leq \bigg(\frac 1 {\Delta^k}\int_{[0,\Delta]^k}|\Mcal_k\y^{\otimes k}( x+t1)|d x\bigg) \e_n\\
&\leq \Big(k\cc ^{k-1}
+\frac{k(k-1)}2\cc ^{k}\frac 1 \Delta\sum_{\ell=1}^m\int_t ^{\Delta+t} K_\ell(x)^2d x\Big)\|y\|_\alpha^k\e_n,
\end{align*}
and $\operatorname{dom}(\mathcal{A})$ is dense in $\B$, we can similarly conclude that $\|\Lcal_k \A_n-\Lcal_k \A_t\|_{*k}\to0$. As a side product of this computation we also get that
\begin{equation}\label{eqn208}
|\langle \Lcal_k \A_t,\y\rangle_\alpha|\leq c_t \|\y\|_\times,\qquad\y\in \B^{\otimes k},
\end{equation}
where $c_t:=(k\cc ^{k-1}
+\frac{k(k-1)}2\cc ^{k}\frac 1 \Delta\sum_{\ell=1}^m\int_t ^{\Delta+t} K_\ell(x)^2d x)\|F\|_{\infty,k}$.

\emph{Claim:} $\partial_t\A_t=\Lcal_k\A_t$ \\
Fix $\y\in \operatorname{dom}(\mathcal{A})$ and define $m_t:=Z_t\y^{\otimes k}$.  Observe that for all $ x\in(0,\Delta]^k$ and $t\in[0,T]$ we have 
\begin{align*}
|\partial_tm_t( x)|
&=|\partial_t(\y^{\otimes k}(x+t1)e^{\int_0^tV_k(x+\tau1)d\tau})|\\
&=|(\sum_{i=1}^k y'(x_i+t)\prod_{j\neq i}\y(x_j+t)+\prod_{i=1}^k\y(x_i+t)V_k(x+t1))e^{\int_0^tV_k(x+\tau1)d\tau}|\\
&\leq(k \|\y'\|_\alpha  \|\y\|_\alpha ^{k-1}\cc ^{k}+\cc ^{k}\overline V_k(x)\|\y\|_\alpha^k)C_T,
\end{align*}
where  $\overline V_k( x):=\sup_{t\in[0,T]}V_k(x+1t)$. Condition \eqref{eqn209} then implies that the map $x\mapsto\sup_{t\in[0,T]}|\partial_tm_t( x)|$ is dominated by an integrable function. Hence,   
 the Leibniz integral rule and \eqref{eqn9} yield
 \begin{align*}
\partial_t\langle \A_t,\y^{\otimes k}\rangle_\alpha
&=\frac 1 {\Delta^k}\int_{[0,\Delta]^k}\partial_tm_t( x) d x\\
&=\frac 1 {\Delta^k}\int_{[0,\Delta]^k}\partial_t\big(\y^{\otimes k}(x+1t)e^{\int_0^tV_k(x+\tau1)d\tau}\big) d x\\
&=\frac 1 {\Delta^k}\int_{[0,\Delta]^k} \Mcal_k\y^{\otimes k}( x +1t)e^{\int_0^tV_k(x+\tau1)d\tau} d x\\
&=\langle \Lcal_k\A_t,\y^{\otimes k}\rangle_\alpha,
 \end{align*}
for each $\y\in \operatorname{dom}(\mathcal{A})$. Since $(\langle \Lcal_k\A_t,\y^{\otimes k}\rangle_\alpha)_{t\geq0}$ is continuous for all  $\y\in \operatorname{dom}(\mathcal{A})$, \eqref{eq:dualPDEBergomi} follows by the fundamental theorem of calculus.
By \eqref{eqn208} this result can be extended to all $\y\in \B$ and the claim follows.

\emph{Claim:} Conditions \ref{itiinew} and \ref{itii} of Theorem~\ref{thm1} hold. \\
We apply Lemma \ref{lem:sufficient}. The moment condition is satisfied by assumption and Condition \eqref{eqn104} holds since $\sup_{t\in[0,T]}c_t<\infty$ with $c_t$ given in \eqref{eqn208}.

The result now follows from Theorem \ref{thm1}.
\end{enumerate} 
\end{proof}

%%%%%%PROOF OF SECOND PROPOSITION
\begin{proof}[Proof of Proposition~\ref{prop1}]Recall that by Remark~\ref{rem:mildconcept}\ref{it1iii}, the process $(\lambda_t)_{t \geq 0}$ solves the martingale problem  for the polynomial operator $L:P^D\to P$.
We follow again the scheme of Section \ref{sec:vix}. Throughout the proof we shall need the following inequalities. Set $C_t:=e^{t\cc ^k\|V_k\|_{\times}}$ and note that $e^{\int_0^tV_k(f(\tau))d\tau}\leq C_t<\infty$ for each $f:[0,t]\to[0,\Delta+t]^k$ and each $t>0$. 
\begin{enumerate}
\item We observe that $M_n\vec y=(\Mcal_0\y_0,\ldots,\Mcal_n\y_n)$ for 
$$
\Mcal_k\y(x)=1^\top\nabla\y(x)+\int \y( x+\xi)\nu(x,d\xi),\qquad \y\in (\operatorname{dom}(\mathcal{A}))^{\otimes k},
$$
where $\mathcal{A}$  is given by \eqref{eqn5}. 

\item 
We now prove  that $m_t=Z_ty(x):=\E[e^{\int_0^tV_k(X^{(k)}_\tau)d\tau}\y(X^{(k)}_t)|X^{(k)}_0=x]$ solves \eqref{eqn204} on $\R_+^k\times [0,T]$ for each $\y\in \R+C^1_0(\R_+^k)$ and thus, since $(\operatorname{dom}(\Acal))^{\otimes k}\subseteq \R+C^1_0(\R_+^k)$,  for each $y\in (\operatorname{dom}(\Acal))^{\otimes k}$.
Indeed, this follows from the Feynman-Kac formula. To this end note that $\Mcal_k\y( x)=\Gcal^{(k)} \y( x)+V_k( x)\y(x)$ and, by the assumptions on $K$, the map $x \mapsto \int (\| \xi \|^2 \wedge 1) \nu(x, d\xi)$ is continuous and bounded. Thus there exists a unique solution to the  martingale problem  for $\Gcal^{(k)}:\R+C^1_0(\R_+^k)\to \R+C_0(\R_+^k)$ (see e.g.~Theorem III.2.34 in \cite{JS:03}). Moreover, by It\^o's formula $Z_ty$ is differentiable with respect ot $t$, whence the claim follows. 

In particular, $(Z_t)_{t \geq 0}$ is a strongly continuous semigroup on $\R+C_0(\R_+^k)$ with generator $\mathcal{M}_k$ implying that 
\begin{align}\label{eq:strongstrong}
\lim_{\e\to0}\frac 1 \e \|Z_{t+\e}y-Z_ty-\e\Mcal_k Z_ty\|_\times=0
\end{align}
for each $y \in \R+C^1_0(\R_+^k)$.
Finally, since $\P(X^{(k)}_t\in[0,\Delta+t]^k | X^{(k)}_0\sim \Ucal([0,\Delta]^k))=1$, by \eqref{eqn14} and \eqref{eqn207}  we get   that
$$
\begin{aligned}
|\langle \widehat \A^{\otimes k},Z_t\y^{\otimes k}\rangle_\alpha|
&=|\E[e^{\int_0^tV_k(X^{(k)}_\tau)d\tau}\y^{\otimes k}(X^{(k)}_t)|X^{(k)}_0\sim \Ucal([0,\Delta]^k)]|
\leq \cc ^kC_t \|\y\|_\alpha^k
\end{aligned}
$$
for each $\y\in \R+C_0(\R_+^k)$, and thus for each $\y\in \operatorname{dom}(\Acal))^{\otimes k}$. This proves \eqref{eq:linfunc}.

\item Due to \eqref{eq:linfunc} and using the notation  introduced in \eqref{eqn14}, we can now define a candidate solution  $\A_t\in (\B^{\otimes k})^*$ (in the sense of Definition \ref{def:sol}) of $\partial_t\A_t=\Lcal_k\A_t$ for $\A_0=\widehat \A^{\otimes k}$ as
 $$\langle \A_t,\y\rangle_\alpha:=\langle\widehat \A^{\otimes k},Z_t \y\rangle_\alpha
 ,\qquad \y\in \B^{\otimes k}.$$
 \item 
In order to conclude the proof we just need to check that $(\A_t)_{t\geq0}$ satisfies the conditions of the dual moment formula as given in  Theorem~\ref{thm1}\ref{iti}-\ref{itii}. Similarly as  in the proof of Proposition \ref{lem7} they consist in the following properties.

\emph{Claim:} $\partial_t\langle\A_t,\y^{\otimes k}\rangle_\alpha=\langle\A_t,\Mcal_k \y^{\otimes k}\rangle_\alpha$\\ 
 Observe that fixing $\y\in \R+C_0^1(\R_+^k)$ and using that $\langle\A_{t+\e},\y\rangle_\alpha=\langle\widehat\A^{\otimes k},Z_{t+\e}\y\rangle_\alpha
=\langle\widehat \A^{\otimes k},Z_{t}Z_\e \y\rangle_\alpha=\langle\A_{t}, Z_\e \y\rangle_\alpha$, by \eqref{eq:strongstrong} we can compute
\begin{align*}
\frac 1 \e|\langle\A_{t+\e},\y\rangle_\alpha-\langle\A_t,\y\rangle_\alpha-\e\langle\A_t,\Mcal_k \y\rangle_\alpha
&=\frac 1 \e|\langle\A_t,Z_{\e}\y\rangle_\alpha-\langle\A_t,\y\rangle_\alpha-\e\langle\A_t,\Mcal_k \y\rangle_\alpha|\\
&\leq\|\A_t\|_{*k}\frac 1 \e\|Z_{\e}\y-\y-\e \Mcal_k \y\|_\times
\to0.
\end{align*}
Hence, $\partial_t\langle\A_t,\y\rangle_\alpha=\langle\A_t,\Mcal_k \y\rangle_\alpha$ for all $y \in \R+C_0^1(\R_+^k)$.

\emph{Claim:} $\A_t\in \Dcal(\Lcal_k)$\\
 Observe that  since $(\y\mapsto  \int \y(x+\xi)\nu(x,d\xi))$ is a bounded linear operator we have that
$$\Dcal(\Lcal_k)=\Dcal({\Acal^*\otimes \Id^{\otimes (k-1)}}):=\{\A\in({\B^{\otimes k}})^*\colon \langle \A,\nabla\y^{\otimes k}\rangle_\alpha\leq C_\A\|\y\|_\alpha^k,\ \y\in\operatorname{dom}(\Acal)\}.$$

Fix $\y\in \R+C_0^1(\R_+)$ and note that 
$$ \langle \A_t,\Mcal_k\y^{\otimes k}\rangle_\alpha
=\partial_t\langle \A_t,{\y^{\otimes k}}\rangle_\alpha
=\partial_t\langle \widehat \A^{\otimes k},Z_t{\y^{\otimes k}}\rangle_\alpha
=\langle \widehat \A^{\otimes k},\Mcal_kZ_t{\y^{\otimes k}}\rangle_\alpha.$$
Since $\Mcal_kZ_t{\y^{\otimes k}}\in \R+C_0(\R_+^k)$ we know that $1^\top\nabla Z_t{\y^{\otimes k}}\in \R+C_0(\R_+^k)$. By the fundamental theorem of calculus we can thus compute
\begin{align*}
\langle \widehat \A^{\otimes k}&,\Mcal_kZ_t{\y^{\otimes k}}\rangle_\alpha
=\langle \widehat \A^{\otimes k},1^\top\nabla Z_t{\y^{\otimes k}}+\int Z_t\y^{\otimes k}(\fdot+\xi)\nu(\fdot,d\xi)\rangle_\alpha\\
&= \frac k \Delta \langle \widehat \A^{\otimes (k-1)}, Z_t{\y^{\otimes k}}(\Delta,\fdot)\rangle_\alpha
-\frac k \Delta \langle \widehat \A^{\otimes (k-1)}, Z_t{\y^{\otimes k}}(0,\fdot)\rangle_\alpha
+ \langle \widehat \A^{\otimes k},\int Z_t\y^{\otimes k}(\fdot+\xi)\nu(\fdot,d\xi)\rangle_\alpha\\
&=
  \frac k \Delta \E[e^{\int_0^tV_k(X^{(k)}_\tau)d\tau}\y^{\otimes k}(X^{(k)}_t)|X^{(k)}_0=\delta_\Delta\otimes \Ucal([0,\Delta]^{k-1})]\\
&-\frac k \Delta\E[e^{\int_0^tV_k(X^{(k)}_\tau)d\tau}\y^{\otimes k}(X^{(k)}_t)|X^{(k)}_0=\delta_0\otimes \Ucal([0,\Delta]^{k-1})] \\
&+\frac{k(k-1)}{2\Delta^2}\Big(\int_0^\Delta K(x)dx\Big)^2\E[e^{\int_0^tV_k(X^{(k)}_\tau)d\tau}\y^{\otimes k}(X^{(k)}_t)|X^{(k)}_0=\delta_0\otimes\delta_0\otimes \Ucal([0,\Delta]^{k-2})].
\end{align*}
An inspection of this expression yields
$$
| \langle \A_t,\Mcal_k\y^{\otimes k}\rangle_\alpha|\leq \Big(\frac{2k}\Delta+\frac{k(k-1)}{2\Delta^2}\int_0^\Delta K(x)^2dx\Big)
C_t \cc ^k\|\y\|^k_\alpha=:c_t\|\y\|^k_\alpha
$$
 for all  $t\in[0,T]$. 
 Since $\operatorname{dom}(\mathcal{A})\subseteq \R+C^1_0(\R_+)$ , the claim follows by noting that
$$\langle \A_t,\nabla\y^{\otimes k}\rangle_\alpha
\leq \bigg(c_t+\|\A_t\|_{*k}\frac{k(k-1)}2\|K\|_\alpha^2\bigg)\|\y\|_\alpha^k.$$

Claim: Conditions \ref{itiinew} and \ref{itii} of Theorem~\ref{thm1} hold.\\
 We again apply Lemma \ref{lem:sufficient}. The moment condition is satisfied by assumption and Condition \eqref{eqn104} holds since $\sup_{t\in[0,T]}c_t<\infty$ with $c_t$ given just above.

\end{enumerate} 

\end{proof}

% ----------------------------------------------------------------
\bibliographystyle{plainnat}

\end{document}